\documentclass[12pt]{article}
\usepackage{amsmath, amssymb,latexsym}
\usepackage{color}
\usepackage{hyperref}

\title{Positivity, complex FIOs, and Toeplitz operators}
\author{Lewis A. Coburn \\\small Department of Mathematics \\\small SUNY at Buffalo
\\\small Buffalo \\\small NY 14260, USA\\\small lcoburn@buffalo.edu \and
Michael Hitrik \\\small Department of Mathematics \\\small University of California
\\\small Los Angeles \\\small CA 90095-1555, USA\\\small hitrik@math.ucla.edu \and
Johannes Sj\"ostrand \\\small IMB, Universit\'e de Bourgogne \\\small 9, Av. A. Savary, BP 47870
\\\small FR-21078 Dijon, France\\\small and UMR 5584 CNRS
\\\small johannes.sjostrand@u-bourgogne.fr}
\date{}

\def\wrtext#1{\relax\ifmmode{\leavevmode\hbox{#1}}\else{#1}\fi}
\def\abs#1{\left|#1\right|}
\def\begeq{\begin{equation}}
\def\endeq{\end{equation}}

\def\iint{\int\hskip -2mm\int}

\def\Re{{\rm Re\,}}
\def\Im{{\rm Im\,}}

\def\part#1{\frac{\partial}{\partial #1}}

\def\norm#1{||\,#1\,||}

\newcommand{\real}{\mbox{\bf R}}
\newcommand{\comp}{\mbox{\bf C}}

\renewcommand{\Re}{\mbox{\rm Re\,}}
\renewcommand{\Im}{\mbox{\rm Im\,}}

\renewcommand{\exp}{\mbox{\rm exp\,}}

\newtheorem{dref}{Definition}[section]

\newtheorem{theo}[dref]{Theorem}
\newtheorem{prop}[dref]{Proposition}

\newenvironment{proof}{\vspace{.3cm}\noindent{{\em Proof:}}}{\hfill$\Box$}

\begin{document}

\maketitle

\vspace*{1cm}
\noindent
{\bf Abstract}: We establish a characterization of complex linear canonical transformations that are positive with respect to a pair of strictly plurisubharmonic quadratic weights. As an application, we show that the boundedness of a class of Toeplitz operators on the Bargmann space is implied by the boundedness of their Weyl symbols.


\vskip 2mm
\noindent
{\bf Mathematics Subject Classifi\-ca\-tion 2000}: 32U05, 32W25, 35S30, 47B35, 70H15

%

\tableofcontents
\section{Introduction and statement of results}
\setcounter{equation}{0}
The notion of a positive complex Lagrangian manifold, introduced by H\"or\-man\-der~\cite{H71}, has long played an important r\^ole in microlocal analysis and spectral theory. Restricting the attention to the linear case, relevant for this work, let us recall that a complex Lagrangian plane $\Lambda \subset \comp^{2n}$ is said to be positive if we have
\begeq
\label{eq1.1}
\frac{1}{i} \sigma(\rho, {\cal C}(\rho)) \geq 0,\quad \rho \in \Lambda.
\endeq
Here $\sigma$ is the complex symplectic form on $\comp^{2n}$ and ${\cal C}: \comp^{2n} \rightarrow \comp^{2n}$ is the antilinear map of complex conjugation.  Let us mention here several familiar problems, where considerations of positive Lagrangian manifolds are essential. These include the spectral analysis and resolvent estimates for elliptic quadratic differential operators~\cite{Sj74},~\cite{HiSjVi}, the study of spectral instability and pseudospectra for semiclassical non-normal operators~\cite{H60},~\cite{DSZ}, as well as the construction of Gaussian beam quasimodes for semiclassical selfadjoint operators of principal type, associated with closed elliptic trajectories~\cite{R76},~\cite{BB}.

\bigskip
\noindent
The work~\cite{Sj82} by the third named author introduced and developed the notion of positivity of a complex Lagrangian space relative to a strictly plurisubharmonic quadratic weight, which is the starting point for the present work. To recall this notion, we let $\Phi_0$ be a strictly plurisubharmonic quadratic form on $\comp^n$ and let us introduce the real linear subspace
\begeq
\label{eq1.2}
\Lambda_{\Phi_0} = \left\{\left(x,\frac{2}{i}\frac{\partial{\Phi_0}}{\partial x}(x)\right),\,\, x\in \comp^n\right\} \subset \comp^{2n}.
\endeq
We can view $\Lambda_{\Phi_0}$ as the image of the real phase space $\real^{2n}$ under a suitable complex linear canonical transformation, and in particular we notice that $\Lambda_{\Phi_0}$ is maximally totally real. In analogy with the discussion above, we say that a complex linear Lagrangian space $\Lambda \subset \comp^{2n}$ is positive relative to $\Lambda_{\Phi_0}$ provided that the natural analog of (\ref{eq1.1}) holds,
\begeq
\label{eq1.3}
\frac{1}{i} \sigma(\rho, \iota_{\Phi_0}(\rho)) \geq 0,\quad \rho \in \Lambda.
\endeq
Here the map of complex conjugation ${\cal C}$ has been replaced by the unique antilinear involution $\iota_{\Phi_0}: \comp^{2n} \rightarrow \comp^{2n}$ such that $\iota_{\Phi_0}|_{\Lambda_{\Phi_0}} = 1$. A result of~\cite{Sj82} establishes a complete characterization of complex Lagrangians that are positive relative to $\Lambda_{\Phi_0}$ --- see also Theorem \ref{theo_first} below.

\medskip
\noindent
In this work, we shall be mainly concerned with positive complex canonical transformations. Indeed, the main goal of the present work is to provide a characterization of positive complex linear canonical transformations relative to plurisubharmonic weights, and to consider Fourier integral operators (FIOs) in the complex domain associated to positive canonical transformations, establishing a link between such operators and Toeplitz operators. In particular, it seems that the point of view of complex FIOs allows us to shed some new light on some basic questions in the theory of Toeplitz operators. We would like to emphasize here that the original motivation for attempting to establish a link between FIOs in the complex domain and Toeplitz operators came from a talk delivered by the first named author at the conference "Complex and functional analysis and their interactions with harmonic analysis", at the Mathematical Research and Conference Center, B\c{e}dlewo, June 2017.

\bigskip
\noindent
We shall now proceed to define the notion of a complex linear canonical transformation which is positive relative to a strictly plurisubharmonic quadratic weight, and to state our main results. In fact, proceeding in the spirit of the discussion above, it will be more transparent to introduce the notion of positivity relative to a pair of strictly plurisubharmonic quadratic forms rather than relative to a single one. Thus, let $\Phi_1$, $\Phi_2$ be two strictly plurisubharmonic quadratic forms on $\comp^n$ with the corresponding antilinear involutions $\iota_{\Phi_1}$, $\iota_{\Phi_2}$. Let $\kappa: \comp^{2n} \rightarrow \comp^{2n}$ be a complex linear canonical transformation, $\kappa^* \sigma = \sigma$. We say that $\kappa$ is positive relative to $(\Lambda_{\Phi_1}, \Lambda_{\Phi_2})$ provided that
\begeq
\label{eq1.4}
\frac{1}{i} \biggl(\sigma(\kappa(\rho), \iota_{\Phi_1} \kappa(\rho)) - \sigma(\rho, \iota_{\Phi_2}(\rho))\biggr) \geq 0,\quad \rho \in \comp^{2n}.
\endeq
The positivity of $\kappa$ relative to $(\Lambda_{\Phi_1},\Lambda_{\Phi_2})$ is said to be strict provided that the inequality in (\ref{eq1.4}) is strict for all $0\neq \rho \in \comp^{2n}$. Let us remark that in the case when the positivity is taken relative to the real phase space $\real^{2n}$, see (\ref{eq1.1}), such canonical transformations were studied in~\cite{H83},~\cite{H95}, see also the recent works~\cite{PRW},~\cite{AV}.

\medskip
\noindent
We can now state the first main result of this work.

\begin{theo}
\label{theo_main}
Let $\kappa: \comp^{2n} \rightarrow \comp^{2n}$ be a complex linear canonical transformation and let $\Phi_1$, $\Phi_2$ be strictly plurisubharmonic quadratic forms on $\comp^{n}$. The canonical transformation $\kappa$ is positive relative to $(\Lambda_{\Phi_1},\Lambda_{\Phi_2})$ precisely when we have \begeq
\label{eq1.5}
\kappa(\Lambda_{\Phi_2}) = \Lambda_{\Phi},
\endeq
where $\Phi$ is a strictly plurisubharmonic quadratic form such that $\Phi \leq \Phi_1$.
\end{theo}

\medskip
\noindent
{\it Remark.} The definition (\ref{eq1.4}) of a positive canonical transformation is a direct adaptation of the corresponding notion of positivity due to H\"ormander~\cite{H83},~\cite{H95}, to the weighted setting. One advantage of the consideration of the general case of a pair of weights $\Phi_1$, $\Phi_2$, is that we can let $\kappa$ be the identity in (\ref{eq1.4}) and get an invariant notion of the positivity of one plurisubharmonic weight compared to another, in view of Theorem \ref{theo_main}.

\bigskip
\noindent
Our second main result is concerned with applications of Theorem \ref{theo_main} to the study of Toeplitz operators in the Bargmann space
$$
H_{\Phi_0}(\comp^n) = L^2(\comp^n, e^{-2\Phi_0} L(dx)) \cap {\rm Hol}(\comp^n),
$$
where $\Phi_0$ is a strictly plurisubharmonic quadratic form on $\comp^n$. Specifically, we shall be concerned with the continuity properties of (in general unbounded) Toeplitz operators of the form
\begeq
\label{eq1.6}
{\rm Top}(e^{2q}) = \Pi_{\Phi_0} \circ e^{2q} \circ \Pi_{\Phi_0}: H_{\Phi_0}(\comp^n) \rightarrow H_{\Phi_0}(\comp^n),
\endeq
where $q$ is a complex-valued quadratic form on $\comp^n$ and
$$
\Pi_{\Phi_0}: L^2(\comp^n, e^{-2\Phi_0} L(dx)) \rightarrow H_{\Phi_0}(\comp^n)
$$
is the orthogonal projection. Sufficient conditions for the boundedness of ${\rm Top}(e^{2q})$ are provided in the following result.

\begin{theo}
\label{T_bound}
Let $\Phi_0$ be a strictly plurisubharmonic quadratic form on $\comp^n$ and let $q$ be a quadratic form on $\comp^n$ such that
\begeq
\label{eq1.7}
2{\rm Re}\, q(x) < \Phi_{\rm herm}(x) := (1/2)\left(\Phi_0(x) + \Phi_0(ix)\right),\quad x \neq 0
\endeq
and
\begeq
\label{eq1.8}
\partial_x \partial_{\overline{x}} \left(\Phi_0 - q\right) \neq 0.
\endeq
Let $a\in C^{\infty}(\Lambda_{\Phi_0})$ be the Weyl symbol of the Toeplitz operator ${\rm Top}(e^{2q})$. Assume that $a \in L^{\infty}(\Lambda_{\Phi_0})$. Then the Toeplitz operator
$$
{\rm Top}(e^{2q}): H_{\Phi_0}(\comp^n) \rightarrow H_{\Phi_0}(\comp^n)
$$
is bounded.
\end{theo}

\medskip
\noindent
{\it Remark}. Let us remark that Theorem \ref{T_bound} is closely related to the conjecture of~\cite{BC94},~\cite{LC}, stating that a Toeplitz operator is bounded on $H_{\Phi_0}(\comp^n)$ precisely when its Weyl symbol is bounded on $\Lambda_{\Phi_0}$. Theorem \ref{T_bound} can therefore be regarded as establishing the sufficiency part of the conjecture in the special case when the Toeplitz symbol is of the form $\exp(2q)$, where $q$ is a complex valued quadratic form on ${\bf C}^n$, satisfying (\ref{eq1.7}), (\ref{eq1.8}).

\medskip
\noindent
{\it Remark}. As we shall see in Section \ref{sect_four}, the condition (\ref{eq1.7}) guarantees that the operator ${\rm Top}(e^{2q})$ is densely defined. Notice also that the Hermitian form $\Phi_{\rm herm}$ in (\ref{eq1.7}) is positive definite on $\comp^n$, thanks to the strict plurisubharmonicity of $\Phi_0$.

\bigskip
\noindent
The plan of the paper is as follows. In Section \ref{sect_two}, we establish the necessity part of Theorem \ref{theo_main}, by means of direct geometric arguments, relying on some general results of~\cite{Sj82}, see also~\cite{CGHS},~\cite{HiSj15}. The proof of Theorem \ref{theo_main} is completed in Section \ref{sect_three}, where we have found it convenient to introduce explicitly a Fourier integral operator in the complex domain quantizing the canonical transformation $\kappa$ satisfying (\ref{eq1.5}), when verifying the positivity of $\kappa$. Applications to Toeplitz operators are given in Section \ref{sect_four}, where Theorem \ref{T_bound} is established. Appendix \ref{appA} is devoted to some elementary remarks concerning integral representations for linear continuous maps between weighted spaces of holomorphic functions, which can be regarded as a version of the Schwartz kernel theorem in this setting. These representations are to be applied in the main text when deriving a Bergman type representation for our complex FIOs. Finally, Appendix \ref{appB}, for the use in Section \ref{sect_four}, characterizes boundedness properties of operators given as Weyl quantizations of symbols of the form $e^{iF(x,\xi)}$, where $F$ is a holomorphic quadratic form on $\comp^{2n}$.

\medskip
\noindent
{\bf Acknowledgements}. The second named author would like to express his sincere and profound gratitude to the Institut de Math\'ematiques de Bourgogne at the Universit\'e de Bourgogne for the kind hospitality in August-September 2017, where part of this project was conducted.

\section{Positive Lagrangian planes and positive canonical transformations in the $H_{\Phi}$--set\-ting}
\label{sect_two}
\setcounter{equation}{0}
\noindent
Let $\Phi_0$ be a strictly plurisubharmonic quadratic form on $\comp^n$. Associated to $\Phi_0$ is the I-Lagrangian R-symplectic linear manifold
$\Lambda_{\Phi_0}$, given by
\begeq
\label{eq2.1}
\Lambda_{\Phi_0} = \left\{\left(x,\frac{2}{i}\frac{\partial \Phi_0}{\partial x}(x)\right);\, x\in \comp^n\right\} \subset \comp^{2n}.
\endeq
The linear manifold $\Lambda_{\Phi_0}$ is maximally totally real, and we let $\iota_{\Phi_0}$ be the unique antilinear involution
\begeq
\label{eq2.2}
\iota_{\Phi_0}: \comp^{2n} \rightarrow \comp^{2n},
\endeq
such that the restriction of $\iota_{\Phi_0}$ to $\Lambda_{\Phi_0}$ is the identity. For future reference, we may recall the explicit description of the
involution $\iota_{\Phi_0}$ given in~\cite{HiSj15},
\begeq
\label{eq2.3}
\left(y, \frac{2}{i} \left(\Phi''_{0,xx} y + \Phi''_{0,x \bar x}\bar x\right) \right) \mapsto
\left(x, \frac{2}{i}\left(\Phi''_{0,xx} x + \Phi''_{0,x \bar x}\bar y\right) \right).
\endeq
We also have
\begeq
\label{eq2.3.1}
\iota_{\Phi_0}: \left(y, \frac{2}{i} \overline{\partial_y \Psi_0(x,\overline{y})}\right) \mapsto \left(x, \frac{2}{i}
\partial_x \Psi_0(x,\overline{y})\right),
\endeq
where $\Psi_0(x,y)$ is the polarization of $\Phi_0$, i.e., the unique holomorphic quadratic form on $\comp^n_x \times \comp^n_y$, such that $\Psi_0(x,\overline{x}) = \Phi_0(x)$.

\medskip
\noindent
Let $\Lambda \subset \comp^{2n}$ be a $\comp$-Lagrangian space, i.e. a complex linear subspace such that ${\rm dim}_{{\bf C}} \Lambda = n$ and $\sigma|_{\Lambda} = 0$. Here $\sigma$ is the standard symplectic form on $\comp^{2n}$. Let us consider the Hermitian form
\begeq
\label{eq2.3.2}
b(\nu, \mu) = \frac{1}{i} \sigma(\nu, \iota_{\Phi_0}(\mu)),\quad \nu, \mu \in \comp^{2n}.
\endeq
We say that $\Lambda$ is positive relative to $\Lambda_{\Phi_0}$ if the Hermitian form (\ref{eq2.3.2}) is positive semidefinite when restricted to $\Lambda$, \begeq
\label{eq2.3.3}
b(\mu,\mu) \geq 0,\quad \mu \in \Lambda.
\endeq
The positivity is said to be strict if the form $b$ in (\ref{eq2.3.2}) is positive definite along $\Lambda$. As remarked in the introduction, this notion is a direct adaptation of the corresponding notion of positivity due to H\"ormander~\cite{H71} where in place of $(\Lambda_{\Phi_0}, \iota_{\Phi_0})$ we have $(\real^{2n},{\cal C})$, with ${\cal C}$ being the antilinear map of complex conjugation.

\medskip
\noindent
{\it Remark.} It is easy to see and is established in~\cite{CGHS},~\cite{HiSj15} that the Hermitian form $b$ is non-degenerate along $\Lambda$ precisely when $\Lambda$ and $\Lambda_{\Phi_0}$ are transversal.

\medskip
\noindent
Our starting point is the following well known result, see~\cite{Sj82},~\cite{CGHS},~\cite{HiSj15}.
\begin{theo}
\label{theo_first}
A $\comp$-Lagrangian space $\Lambda$ is positive relative to $\Lambda_{\Phi_0}$ if and only if $\Lambda = \Lambda_{\Psi}$, where $\Psi$ is a pluriharmonic quadratic form such that $\Psi \leq \Phi_0$.
\end{theo}

\medskip
\noindent
The proof of Theorem \ref{theo_first} given in~\cite{Sj82},~\cite{CGHS},~\cite{HiSj15} discusses the case of strictly positive Lagrangian planes only and depends on the general fact that the set of all $\comp$-Lagrangian spaces which are strictly positive relative to $\Lambda_{\Phi_0}$ is a connected component in the set of all $\comp$-Lagrangian spaces that are transversal to $\Lambda_{\Phi_0}$. Here we shall give a more direct proof, using the explicit description of the involution $\iota_{\Phi_0}$, given in (\ref{eq2.3}), (\ref{eq2.3.1}). Let $\Lambda\subset \comp^{2n}$ be $\comp$-Lagrangian, positive relative to $\Lambda_{\Phi_0}$. It follows from (\ref{eq2.3}), as explained in~\cite{Sj82},~\cite{HiSj15}, that the fiber $\{(0,\xi); \xi\in \comp^n\}$ is strictly negative relative to $\Lambda_{\Phi_0}$, in the sense that the Hermitian form $b$ in (\ref{eq2.3.2}) is negative definite along the fiber, and therefore $\Lambda$ is necessarily of the form $\xi = \partial_x \varphi(x)$, where $\varphi$ is a holomorphic quadratic form on $\comp^n$. It follows that
\begeq
\label{eq2.3.4}
\Lambda = \Lambda_{\Psi},
\endeq
where $\Psi = -{\rm Im}\, \varphi$ is pluriharmonic quadratic. We shall now see that $\Psi \leq \Phi_0$, and to this end, let us consider the decomposition,
\begeq
\label{eq2.3.5}
\Phi_0 = \Phi_{\rm herm} + \Phi_{\rm plh},
\endeq
where
\begeq
\label{eq2.3.6}
\Phi_{\rm herm}(x) = \Phi''_{0, \overline{x}x} x \cdot \overline{x}
\endeq
is positive definite Hermitian and
\begeq
\label{eq2.3.7}
\Phi_{\rm plh}(x) = {\rm Re}\, \left(\Phi''_{0,xx} x \cdot x\right)
\endeq
is pluriharmonic. Let
$$
A = \frac{2}{i} \left(\Phi_{\rm plh}\right)''_{x x} = \frac{2}{i} \left(\Phi_0\right)''_{x x},
$$
and let us consider the complex linear "vertical" canonical transformation
\begeq
\label{eq2.3.8}
\kappa_A(y,\eta) = (y, \eta + Ay).
\endeq
We have
\begeq
\label{eq2.3.9}
\kappa_A(\Lambda_{\Phi_{\rm herm}}) = \Lambda_{\Phi_0},
\endeq
and letting $\iota_{\Phi_{\rm herm}}$ be the antilinear involution associated to $\Lambda_{\Phi_{\rm herm}}$, it is then clear that
\begeq
\label{eq2.3.901}
\iota_{\Phi_{\rm herm}} = \kappa_A^{-1} \circ \iota_{\Phi_0} \circ \kappa_A.
\endeq
It follows that $\Lambda$ is positive relative to $\Lambda_{\Phi_0}$ precisely when
$$
\kappa_A^{-1}(\Lambda) = \Lambda_{\Psi - \Phi_{\rm plh}}
$$
is positive relative to $\Lambda_{\Phi_{\rm herm}}$, and when proving Theorem \ref{theo_first} we may assume therefore that the pluriharmonic part of $\Phi_0$ vanishes. In this discussion, we are also allowed to perform complex linear changes of variables in $\comp^n$, which correspond to canonical transformations of the form $\kappa_C: (y,\eta) \mapsto (C^{-1}y, C^t \eta)$, where $C$ is an invertible complex $n \times n$ matrix. We have $\kappa_C(\Lambda_{\Phi_0}) = \Lambda_{\Phi_1}$, $\Phi_1(x) = \Phi_0(Cx)$, and it follows therefore that when establishing Theorem \ref{theo_first} it suffices to consider the model case when
\begeq
\label{eq2.3.91}
\Phi_0(x) = \frac{\abs{x}^2}{2}.
\endeq
An application of (\ref{eq2.3}) shows that the involution $\iota_{\Phi_0}$ is then given by
\begeq
\label{eq2.3.92}
(y,\eta) \mapsto (\frac{1}{i} \overline{\eta}, \frac{1}{i} \overline{y}),
\endeq
and therefore
\begeq
\label{eq2.3.93}
b(\mu,\mu) = \frac{1}{i} \sigma(\mu, \iota_{\Phi_0}(\mu)) = \abs{x}^2 - \abs{\xi}^2,\quad \mu = (x,\xi) \in \comp^{2n}.
\endeq
When $\mu \in \Lambda = \Lambda_{\Psi}$, we write $\xi = (2/i) \partial_x \Psi(x) = \partial_x \varphi(x)$, $\Psi(x) = -{\rm Im}\, \varphi$, where $\varphi$ is a quadratic holomorphic form, and therefore if $\Lambda$ is positive relative to $\Lambda_{\Phi_0}$, then (\ref{eq2.3.93}) shows that
\begeq
\label{eq2.3.94}
\abs{\varphi''_{xx} x} \leq \abs{x},\quad x \in \comp^n \Longleftrightarrow \norm{\varphi''_{xx}} \leq 1.
\endeq
We get
\begeq
\label{eq2.3.95}
\Psi(x) = -{\rm Im}\, \varphi(x) \leq \frac{\abs{\varphi''_{xx} x\cdot x}}{2} \leq \frac{\abs{x}^2}{2} = \Phi_0(x),\quad x\in \comp^n.
\endeq
Conversely, let $\Lambda$ be $\comp$-Lagrangian of the form $\Lambda = \Lambda_{\Psi}$, where $\Psi$ is pluriharmonic quadratic such that $\Psi \leq \Phi_0$. Let us write $\Psi = -{\rm Im}\, \varphi$, where $\varphi$ is a holomorphic quadratic form. We shall now see that $\Lambda_{\Psi}$ is positive relative to $\Lambda_{\Phi_0}$, and it follows from the remarks above that it suffices to verify the positivity in the model case when $\Phi_0$ is given by (\ref{eq2.3.91}), so that we have
\begeq
\label{eq2.3.96}
\Psi(x) = -{\rm Im}\, \varphi(x) \leq \Phi_0(x) = \frac{\abs{x}^2}{2}.
\endeq
Writing
\begeq
\label{eq2.3.97}
-{\rm Im}\, \varphi''_{xx} x\cdot x \leq \abs{x}^2,
\endeq
replacing $x$ by $e^{i\theta}x$ and varying $\theta \in \real$, we get
\begeq
\label{eq2.3.98}
\abs{\varphi''_{xx} x\cdot x} \leq \abs{x}^2, \quad x\in \comp^n.
\endeq
Next, writing
$$
\varphi''_{xx} x\cdot y = \frac{1}{4}\left(\varphi''_{xx} (x+y)\cdot (x+y) - \varphi''_{xx} (x-y)\cdot (x-y)\right),
$$
we get, using (\ref{eq2.3.98}),
\begeq
\label{eq2.3.99}
\abs{\varphi''_{xx} x\cdot y} \leq \frac{1}{4}\left(\abs{x+y}^2 + \abs{x-y}^2\right) = \frac{1}{2} \left(\abs{x}^2 + \abs{y}^2\right).
\endeq
Replacing $x\mapsto \lambda^{1/2}x$, $y \mapsto \lambda^{-1/2}y$, $\lambda > 0$, we get
\begeq
\label{eq2.3.991}
\abs{\varphi''_{xx} x\cdot y} \leq \frac{1}{2} \left(\lambda \abs{x}^2 + \frac{1}{\lambda} \abs{y}^2\right),
\endeq
and choosing $\lambda = \abs{y}/\abs{x}$, assuming for simplicity that $x\neq 0$, $y\neq 0$, we obtain that
$$
\abs{\varphi''_{xx} x\cdot y} \leq \abs{x} \abs{y}.
$$
Hence, $\norm{\varphi''_{xx}} \leq 1$ and the positivity of $\Lambda_{\Psi}$ relative to $\Lambda_{\Phi_0}$ follows from (\ref{eq2.3.93}), (\ref{eq2.3.94}). The proof of Theorem~\ref{theo_first} is complete.

\bigskip
\noindent
{\it Remark.} Closely related to the proof of Theorem~\ref{theo_first} given above is the normal form for strictly plurisubharmonic quadratic forms, given in Lemma 5.1 of~\cite{H97}, see also~\cite{HW}.

\bigskip
\noindent
Let $\Phi_1$, $\Phi_2$ be two strictly plurisubharmonic quadratic forms on $\comp^n$ and let $\kappa: \comp^{2n} \rightarrow \comp^{2n}$ be a complex linear
canonical transformation which is positive relative to $(\Lambda_{\Phi_1},\Lambda_{\Phi_2})$, in the sense of (\ref{eq1.4}). In the remainder of this section, we shall establish the necessity part of Theorem \ref{theo_main}, while the sufficiency is discussed in Section \ref{sect_three}. To this end, let us observe first that the linear I-Lagrangian R-symplectic manifold $\kappa(\Lambda_{\Phi_2})$ is transversal to the fiber $\{(0,\xi); \, \xi \in \comp^n\}$. Indeed, we have in view of (\ref{eq1.4}),
\begeq
\label{eq2.6}
\frac{1}{i} \sigma(\rho, \iota_{\Phi_1}(\rho)) \geq 0,\quad \rho \in \kappa(\Lambda_{\Phi_2}),
\endeq
while, as recalled above, we know from~\cite{Sj82},~\cite{HiSj15} that the fiber is strictly negative relative to $\Lambda_{\Phi_1}$. It follows that $\kappa(\Lambda_{\Phi_2}) = \Lambda_{\Phi}$, where $\Phi$ is a real quadratic form such that the Levi form $\overline{\partial} \partial \Phi$ is non-degenerate. When verifying that $\Phi$ is (necessarily strictly) plurisubharmonic, we claim that it suffices to do so when the pluriharmonic part of $\Phi_2$ vanishes. Indeed, introducing the decomposition (\ref{eq2.3.5}), with the quadratic form $\Phi_2$ in place of $\Phi_0$ and considering the canonical transformation $\kappa_A$ given in (\ref{eq2.3.8}), we see, using also (\ref{eq2.3.901}), that $\kappa$ is positive relative to $(\Lambda_{\Phi_1},\Lambda_{\Phi_2})$ precisely when $\kappa_A^{-1}\circ \kappa \circ \kappa_A$ is positive relative to
$(\Lambda_{\Phi_1 - \Phi_{2,{\rm plh}}}, \Lambda_{\Phi_{2,{\rm herm}}})$. Here $\Phi_{2,{\rm plh}}$ and $\Phi_{2,{\rm herm}}$ are the pluriharmonic and the Hermitian parts of $\Phi_2$, respectively. Here it is also helpful to notice that
$$
\iota_{\Phi_1 - \Phi_{2,{\rm plh}}} = \kappa_A^{-1} \circ \iota_{\Phi_1} \circ \kappa_A.
$$
To summarize, if we know that the generating function of the linear I-Lagrangian R-symplectic manifold
$$
\kappa_A^{-1}\circ \kappa \circ \kappa_A(\Lambda_{\Phi_{2,{\rm herm}}})
$$
is plurisubharmonic, then the same property is also enjoyed by the generating function of $\kappa(\Lambda_{\Phi_2})$. In what follows we shall assume therefore that
\begeq
\label{eq2.9}
\Phi_{2,xx} = \Phi_{2,\overline{x}\,\overline{x}} = 0.
\endeq
As above, in this discussion, we are also allowed to perform complex linear changes of variables in $\comp^n$, which correspond to canonical transformations of the form $(y,\eta) \mapsto (C^{-1}y, C^t \eta)$, where $C$ is an invertible complex $n \times n$ matrix. Such canonical transformations preserve the plurisubharmonicity of the generating functions, and similarly to the proof of Theorem \ref{theo_first}, it suffices therefore to consider the case when
\begeq
\label{eq2.10}
\Phi_2(x) = \frac{\abs{x}^2}{2}.
\endeq
Theorem \ref{theo_first} then shows that the $\comp$-Lagrangian plane given by $\{(x,\xi)\in \comp^{2n};\, \xi =0\}$ is strictly positive relative to $\Lambda_{\Phi_2}$, and therefore $\kappa(\{(x,\xi)\in \comp^{2n};\, \xi = 0\})$ is strictly positive relative to $\Lambda_{\Phi_1}$, in view of the positivity of $\kappa$. Another application of Theorem \ref{theo_first} gives that
\begeq
\label{eq2.13}
\kappa(\{(x,\xi)\in \comp^{2n};\, \xi = 0\}) = \Lambda_{\Psi},
\endeq
where the quadratic form $\Psi$ is pluriharmonic, with $\Psi \leq \Phi_1$.

\medskip
\noindent
Let $\phi(x,y,\theta)$ be a holomorphic quadratic form on $\comp^n_x \times \comp^n_y \times \comp_{\theta}^N$, which is a non-degenerate phase function in the sense of H\"ormander, generating the graph of $\kappa$. It follows from (\ref{eq2.13}), as explained in~\cite{CGHS}, that the quadratic form
\begeq
\label{eq2.14}
\comp^n \times \comp^N \ni (y,\theta) \mapsto -{\rm Im}\, \phi(0,y,\theta)
\endeq
is non-degenerate, and since it is pluriharmonic, the signature is necessarily $(n+N,n+N)$. Recalling that
\begeq
\label{eq2.15}
\kappa(\Lambda_{\Phi_2}) = \Lambda_{\Phi},
\endeq
we see, using~\cite{CGHS}, that the quadratic form
\begeq
\label{eq2.16}
(y,\theta) \mapsto -{\rm Im}\, \phi(0,y,\theta) + \Phi_2(y)
\endeq
is non-degenerate as well. We would like to conclude that the signature of the quadratic form in (\ref{eq2.16}) is also $(n+N,n+N)$, and to that end,
we follow~\cite{Sj82} and consider the continuous deformation
\begeq
\label{eq2.17}
[0,1] \ni t \mapsto -{\rm Im}\, \phi(0,y,\theta) + t \Phi_2(y).
\endeq
Using (\ref{eq2.3.93}) we see that
\begeq
\label{eq2.18}
\frac{1}{i} \sigma (\mu, \iota_{\Phi_2}(\mu)) \geq 0,\quad \mu \in \Lambda_{t\Phi_2},\quad 0 \leq t \leq 1.
\endeq
It follows as before that the I-Lagrangian manifold $\kappa(\Lambda_{t \Phi_2})$ is transversal to the fiber, $0\leq t \leq 1$, and therefore we
conclude that the non-degeneracy of the quadratic forms in (\ref{eq2.17}) is maintained along the deformation $0\leq t \leq 1$. Recalling that the set
of non-degenerate quadratic forms of a fixed given signature is a connected component in the set of all non-degenerate quadratic forms, we conclude that the signature of the quadratic form in (\ref{eq2.16}) is $(n+N,n+N)$. Now, as explained in~\cite{CGHS}, the quadratic form $\Phi$ in (\ref{eq2.15}) is given by
\begeq
\label{eq2.19}
\Phi(x) = {\rm vc}_{y,\theta}(-{\rm Im}\, \phi(x,y,\theta) + \Phi_2(y))
\endeq
where ${\rm vc}_{y,\theta}$ stands for the critical value with respect to $y$, $\theta$, and we conclude by the fundamental lemma of~\cite{Sj82} that $\Phi$ is plurisubharmonic. (As already observed, the plurisubharmonicity of $\Phi$ is necessarily strict.)

\bigskip
\noindent
We shall next see that $\Phi \leq \Phi_1$, and when doing so it will be convenient the discuss the following auxiliary result first, which may be of some independent interest.
\begin{prop}
\label{gen_f}
Let $\kappa: \comp^{2n} \rightarrow \comp^{2n}$ be a complex linear canonical transformation which is positive relative to $(\Lambda_{\Phi_1},\Lambda_{\Phi_2})$. If $\Phi_2$ is strictly convex then $\kappa$ has a generating function $\varphi(x,\eta)$ which is a holomorphic quadratic form such that
\begeq
\label{eq2.19.1}
\kappa: (\varphi'_{\eta}(x,\eta), \eta) \mapsto (x,\varphi'_x(x,\eta)).
\endeq
\end{prop}

\begin{proof}
It suffices to show that the map
$$
\pi: {\rm graph}(\kappa) \ni (x,\xi; y,\eta) \mapsto (x,\eta) \in \comp^{2n}
$$
is bijective, i.e. injective. Let $(0,\xi; y,0) \in {\rm Ker}(\pi)$ so that $\kappa: (y,0) \mapsto (0,\xi)$. Let us consider the Hermitian forms,
$$
b_j(\nu, \mu) = \frac{1}{i} \sigma (\nu, \iota_{\Phi_j}(\mu)),\quad j=1,2.
$$
The strict convexity of $\Phi_2$ together with Theorem \ref{theo_first} implies that
\begeq
\label{eq2.19.2}
b_2((y,0), (y,0)) \asymp \abs{y}^2,\quad y \in \comp^n,
\endeq
and the strict negativity of the fiber with respect to $\Lambda_{\Phi_1}$ gives,
$$
b_1((0, \xi), (0,\xi)) \asymp - \abs{\xi}^2,\quad \xi \in \comp^n.
$$
Hence by the positivity of $\kappa$, we get
$$
0\leq b_1((0, \xi), (0,\xi)) - b_2((y,0), (y,0))  \asymp - \left(\abs{\xi}^2 + \abs{y}^2\right).
$$
It follows that $(y,\xi) = 0$ and we conclude that $\pi$ is injective.
\end{proof}

\medskip
\noindent
{\it Remark}. Assume that the assumptions of Proposition \ref{gen_f} hold. The holomorphic quadratic form $\varphi(x,\theta) - y\cdot \theta$ is then a non-degenerate phase function generating the graph of $\kappa$.

\medskip
\noindent
Let us now turn to the proof of the fact that
\begeq
\label{eq2.20}
\Phi \leq \Phi_1.
\endeq
It follows from the remarks above that it suffices to verify (\ref{eq2.20}) when the pluriharmonic part of $\Phi_2$ vanishes, and since we are again allowed to perform complex linear changes of variables in $\comp^n$, as before, we conclude that it suffices to consider the case when $\Phi_2$ is given by (\ref{eq2.10}). Proposition \ref{gen_f} applies and there exists therefore a holomorphic quadratic form $\varphi(x,\theta)$ such that
\begeq
\label{eq2.21}
\kappa: (\varphi'_{\theta}(x,\theta), \theta) \mapsto (x,\varphi'_x(x,\theta)).
\endeq
We shall now express the positivity of $\kappa$ relative to $(\Lambda_{\Phi_1}, \Lambda_{\Phi_2})$ in terms of the generating function $\varphi$. To this end, we shall first obtain an explicit expression for the Hermitian form
$$
\frac{1}{i} \sigma((y,\eta), \iota_{\Phi_1}(y,\eta)),\quad (y,\eta)\in \comp^{2n},
$$
where we write
\begeq
\label{eq2.21.1}
\Phi_1(x) = \frac{1}{2} L \overline{x} \cdot x + {\rm Re}\, (Ax\cdot x), \quad L = 2 \Phi''_{1,x\overline{x}}, \quad A = \Phi''_{1,xx}.
\endeq
Here $L$ is Hermitian positive definite and performing a unitary transformation, we may assume, for simplicity, that $L$ is diagonal, with real positive diagonal elements. A simple computation using (\ref{eq2.3}) shows that
\begeq
\label{eq2.21.2}
\frac{1}{i} \sigma((y,\eta), \iota_{\Phi_1}(y,\eta)) = L \overline{y} \cdot y + (2Ay - i\eta)\cdot x,
\endeq
where
$$
L \overline{x} = i \eta - 2Ay,
$$
and therefore we get
\begeq
\label{eq2.21.3}
\frac{1}{i} \sigma((y,\eta), \iota_{\Phi_1}(y,\eta)) = L \overline{y} \cdot y - L^{-1} (2iAy + \eta) \cdot \overline{(2iAy + \eta)}.
\endeq
Using also (\ref{eq2.21}), we conclude that $\kappa$ is positive relative to $(\Lambda_{\Phi_1}, \Lambda_{\Phi_2})$ precisely when
\begeq
\label{eq2.23}
L^{-1} (\varphi'_x + 2iAx)\cdot \overline{(\varphi'_x + 2iAx)} + \abs{\varphi'_{\theta}(x,\theta)}^2 \leq L \overline{x}\cdot x + \abs{\theta}^2, \quad (x,\theta) \in \comp^{2n}.
\endeq
It is now easy to conclude the proof of the necessity part of Theorem \ref{theo_main}, using (\ref{eq2.23}). It follows from (\ref{eq2.19}) that we can write
\begeq
\label{eq2.24}
\Phi(x) = {\rm vc}_{y,\theta} \left(-{\rm Im}\, (\varphi(x,\theta) - y\cdot \theta) + \Phi_2(y)\right).
\endeq
At the unique critical point $(y(x), \theta(x))$, we have
\begeq
\label{eq2.25}
y = \varphi'_{\theta}(x,\theta),
\endeq
\begeq
\label{eq2.26}
\frac{2}{i} \frac{\partial \Phi_2}{\partial y}(y) = \theta \Longleftrightarrow \theta = \frac{1}{i} \overline{y}.
\endeq
Injecting (\ref{eq2.26}) into (\ref{eq2.24}), we get
\begeq
\Phi(x) = -{\rm Im}\, \varphi(x,\theta) - \frac{\abs{\theta}^2}{2}, \quad \theta = \theta(x),
\endeq
and in view of (\ref{eq2.21.1}), it suffices therefore to establish the inequality
\begeq
\label{eq2.27}
-2 {\rm Im}\, \varphi(x,\theta) \leq L \overline{x}\cdot x + \abs{\theta}^2 + 2 {\rm Re}(Ax \cdot x), \quad (x,\theta)\in \comp^{2n}.
\endeq
When verifying (\ref{eq2.27}), we write, using the Euler homogeneity relation,
\begeq
\label{eq2.28}
2 \varphi(x,\theta) = \varphi'_x(x,\theta)\cdot x + \varphi'_{\theta}(x,\theta)\cdot \theta,
\endeq
and therefore,
\begeq
\label{eq2.28.1}
-2 {\rm Im}\, \varphi(x,\theta) = - {\rm Im}\left((\varphi'_x(x,\theta) + 2iAx)\cdot x + \varphi'_{\theta}(x,\theta)\cdot \theta\right) + 2 {\rm Re}(Ax \cdot x).
\endeq
An application of the Cauchy-Schwarz inequality with respect to the positive definite Hermitian forms $(x,y) \mapsto L^{-1}x\cdot \overline{y}$, $(x,y) \mapsto x\cdot \overline{y}$ together with the inequality $ab \leq a^2/2 + b^2/2$ allows us to conclude that the first term in the right hand side of (\ref{eq2.28.1}) does not exceed
$$
\frac{1}{2} \left(L^{-1} (\varphi'_x + 2iAx)\cdot \overline{(\varphi'_x + 2iAx)} + L\overline{x}\cdot x + \abs{\varphi'_{\theta}(x,\theta)}^2 + \abs{\theta}^2\right).
$$
The inequality (\ref{eq2.27}) follows, in view of (\ref{eq2.23}). The proof of the necessity part of Theorem \ref{theo_main} is complete.

\bigskip
\noindent
{\it Remark.} In the context of Theorem \ref{theo_main}, assume that $\Phi_1 = \Phi_2 = : \Phi_0$ and let us write
\begeq
\label{eq2.29}
\Phi_0(x) = {\rm sup}_{y\in {\bf R}^n} \left(-{\rm Im}\, \varphi(x,y)\right),
\endeq
where $\varphi(x,y)$ is a holomorphic quadratic form on $\comp^n_x \times \comp^n_y$, such that ${\rm det}\, \varphi''_{xy} \neq 0$ and
${\rm Im}\, \varphi''_{yy} > 0$. In the special case when $\Phi_0$ is given by (\ref{eq2.10}), we can take
$$
\varphi(x,y) = i\left(\frac{x^2}{2} + \sqrt{2} x \cdot y + \frac{y^2}{2}\right).
$$
The complex canonical transformation
\begeq
\label{eq2.30}
\kappa_{\varphi}: \comp^{2n} \ni (y,-\varphi'_y(x,y)) \mapsto (x,\varphi'_x(x,y)) \in \comp^{2n}
\endeq
maps $\real^{2n}$ bijectively onto $\Lambda_{\Phi_0}$, see~\cite{HiSj15}, and it exchanges the complex conjugation map ${\cal C}$ and the involution $\iota_{\Phi_0}$. Setting
\begeq
\label{eq2.31}
\widetilde{\kappa} = \kappa_{\varphi}^{-1} \circ \kappa \circ \kappa_{\varphi},
\endeq
we see that the complex linear canonical transformation $\widetilde{\kappa}$ is positive in the sense of~\cite{H95},
\begeq
\label{eq2.32}
\frac{1}{i}\biggl(\sigma(\widetilde{\kappa}(\rho), {\cal C} \widetilde{\kappa}(\rho)) - \sigma(\rho, {\cal C}(\rho))\biggr) \geq 0, \quad \rho \in \comp^{2n}.
\endeq
An application of Proposition 5.10 of~\cite{H95} allows us to conclude therefore that the map $\widetilde{\kappa}$ enjoys the following factorization,
\begeq
\label{eq2.33}
\widetilde{\kappa} = \widetilde{\kappa_1} \circ \widetilde{\kappa_2} \circ \widetilde{\kappa_3},
\endeq
where $\widetilde{\kappa_1}$ and $\widetilde{\kappa_3}$ are real linear canonical maps and the map $\widetilde{\kappa_2}$ is of the form
\begeq
\label{eq2.34}
\widetilde{\kappa_2} = \exp(-i H_{\widetilde{q}})
\endeq
where $\widetilde{q}$ is a quadratic form with ${\rm Re}\, \widetilde{q} \geq 0$ on $\real^{2n}$ --- see also the discussion in the proof of Proposition 5.12 of~\cite{H95}. We obtain the factorization
\begeq
\label{eq2.35}
\kappa = \kappa_1 \circ \kappa_2 \circ \kappa_3,
\endeq
where we have
\begeq
\label{eq2.36}
\kappa_j: \Lambda_{\Phi_0} \rightarrow \Lambda_{\Phi_0},\quad j =1,3,
\endeq
and
\begeq
\label{eq37}
\kappa_2 = \exp(-i H_q),
\endeq
where $q$ is a holomorphic quadratic form on $\comp^{2n}$ such that ${\rm Re}\, q \geq 0$ along $\Lambda_{\Phi_0}$. The representation (\ref{eq2.35}) can be used to give an alternative proof of the basic inequality $\Phi \leq \Phi_0$ in Theorem \ref{theo_main}, in this special case.

\section{Positivity and Fourier integral operators}
\label{sect_three}
\setcounter{equation}{0}

The purpose of this section is to establish the sufficiency part of Theorem \ref{theo_main}. To this end, let $\Phi_1$, $\Phi_2$ be two strictly plurisubharmonic quadratic forms on $\comp^n$ and let $\kappa: \comp^{2n} \rightarrow \comp^{2n}$ be a complex linear canonical transformation. Assume that \begeq
\label{eq3.1}
\kappa(\Lambda_{\Phi_2}) = \Lambda_{\Phi},
\endeq
where $\Phi$ is a strictly plurisubharmonic quadratic form such that
\begeq
\label{eq3.1.1}
\Phi \leq \Phi_1.
\endeq
We shall establish the positivity of $\kappa$ relative to $(\Lambda_{\Phi_1}, \Lambda_{\Phi_2})$ by making a judicious choice of a non-degenerate phase function generating the graph of $\kappa$, and to this end, it will be convenient to consider a metaplectic  Fourier integral operator associated to $\kappa$. Let therefore $\varphi(x,y,\theta)$ be a holomorphic quadratic form on $\comp^n_x \times \comp^n_y \times \comp^N_{\theta}$, which is a non-degenerate phase function in the sense of H\"ormander, generating the graph of $\kappa$. It follows from~\cite{CGHS} that the plurisubharmonic quadratic form
\begeq
\label{eq3.2}
\comp^n \times \comp^N \ni (y,\theta) \mapsto -{\rm Im}\, \varphi(0,y,\theta) + \Phi_2(y)
\endeq
is non-degenerate of signature $(n+N,n+N)$. We conclude, following~\cite{Sj82},~\cite{CGHS} that the Fourier integral operator
\begeq
\label{eq3.3}
Au(x) = \int\!\!\!\int e^{i\varphi(x,y,\theta)} a u(y)\, dy\, d\theta, \quad a \in \comp,
\endeq
quantizing $\kappa$, can be realized by means of a good contour and we obtain a bounded linear map,
\begeq
\label{eq3.4}
A: H_{\Phi_2}(\comp^n) \rightarrow H_{\Phi}(\comp^n).
\endeq
Here
$$
H_{\Phi_2}(\comp^n) = {\rm Hol}(\comp^n) \cap L^2(\comp^n, e^{-2\Phi_2} L(dx)),
$$
with $H_{\Phi}(\comp^n)$ having an analogous definition.

\medskip
\noindent
We shall now discuss a Bergman type representation of the bounded operator in (\ref{eq3.4}), see also~\cite{MeSj2} for a related discussion. To this end, let us recall from Theorem \ref{theo_kernel} that we can write
\begeq
\label{eq3.5}
Au(x) = \int K_A(x,\overline{y}) u(y)\, e^{-2\Phi_2(y)}\, L(dy) =:\widetilde{A}u(x).
\endeq
Here the kernel $K_A(x,z)$ is holomorphic on $\comp^n_x \times \comp^n_z$, with
$$
y \mapsto \overline{K(x,\overline{y})} \in H_{\Phi_2}(\comp^n),
$$
uniquely determined by (\ref{eq3.5}). If $u\in L^2_{\Phi_2}(\comp^n) = L^2 (\comp^n, e^{-2\Phi_2}L(dx))$ is orthogonal $H_{\Phi_2}(\comp^n)$, we see from (\ref{eq3.5}) that $\widetilde{A} u = 0$. Hence the operator $\widetilde{A}$ in (\ref{eq3.5}) is a well defined linear continuous map
$$
\widetilde{A}: L^2_{\Phi_2}(\comp^n) \rightarrow H_{\Phi_2}(\comp^n).
$$
Furthermore, $\widetilde{A}$ extends to a map: ${\cal E}'(\comp^n) \rightarrow {\rm Hol}(\comp^n)$ and we have
\begeq
\label{eq3.51}
K_A(x,\overline{y}) e^{-2\Phi_2(y)} = \left(\widetilde{A}\delta_y\right)(x),
\endeq
where $\delta_y \in {\cal E}'(\comp^n)$ is the delta function at $y$. Let next $\Pi_2: L^2_{\Phi_2}(\comp^n) \rightarrow H_{\Phi_2}(\comp)$ be the orthogonal projection and let us recall from~\cite{HiSj15} that the operator $\Pi_2$ is given by
\begeq
\label{eq3.52}
\Pi_2 u(x) = a_2 \int e^{2\Psi_2(x,\overline{y}) - \Phi_2(y)} u(y)\, L(dy),\quad a_2 > 0.
\endeq
Here $\Psi_2$ is the polarization of $\Phi_2$, i.e. a holomorphic quadratic form on $\comp^{2n}_{x,y}$ such that $\Psi_2(x,\overline{x}) = \Phi_2(x)$.
We get $\widetilde{A} \delta_y = \widetilde{A} \Pi_2 \delta_y = A \Pi_2 \delta_y$, and it follows from (\ref{eq3.51}) that
\begeq
\label{eq3.6}
K_A(x,\overline{y}) = A(a_2 e^{2\Psi_2(\cdot, \overline{y})})(x).
\endeq
From~\cite{HiSj15}, let us recall the basic property,
$$
2{\rm Re}\, \Psi_2(x,\overline{y}) - \Phi_2(x) - \Phi_2(y) \sim -\abs{x-y}^2,
$$
on $\comp^n_x \times \comp^n_y$, and in particular we have
\begeq
\label{eq3.7}
2 {\rm Re}\, \Psi_2(x,\overline{y}) \leq \Phi_2(x) + \Phi_2(y).
\endeq
It follows that
\begeq
\label{eq3.8}
-{\rm Im}\, \varphi(0,\widetilde{y},\theta) + 2 {\rm Re}\,\Psi_2(\widetilde{y},0) \leq -{\rm Im}\, \varphi(0,\widetilde{y},\theta) + \Phi_2(\widetilde{y}).
\endeq
Here, as observed in (\ref{eq3.2}), the right hand side is a non-degenerate plurisubharmonic quadratic form of signature $(n+N,n+N)$, and since the left hand side is pluriharmonic, we conclude that it is also non-degenerate of signature $(n+N,n+N)$. Writing
$$
-{\rm Im}\, \varphi(0,\widetilde{y},\theta) + 2 {\rm Re}\,\Psi_2(\widetilde{y},0) = {\rm Re}\,\left(i\varphi(0,\widetilde{y},\theta) + 2 \Psi_2(\widetilde{y},0)\right),
$$
we conclude that the holomorphic quadratic form
$$
\comp^n \times \comp^N \ni (\widetilde{y},\theta) \mapsto i\varphi(0,\widetilde{y},\theta) + 2 \Psi_2(\widetilde{y},0)
$$
is non-degenerate. It follows that the holomorphic function
$$
\comp^n \times \comp^N \ni (\widetilde{y},\theta) \mapsto i\varphi(x,\widetilde{y},\theta) + 2 \Psi_2(\widetilde{y},z)
$$
has a unique critical point which is non-degenerate, for each $(x,z)\in \comp^n \times \comp^n$. An application of exact (quadratic) stationary phase allows us therefore to conclude that
\begeq
\label{eq3.9}
K_A(x,\overline{y}) = \widehat{a} e^{2\Psi(x,\overline{y})},\quad \widehat{a}\in \comp.
\endeq
Here $\Psi(x,z)$ is a holomorphic quadratic form on $\comp^{2n}$ given by
\begeq
\label{eq3.10}
2 \Psi(x,z) = {\rm vc}_{\widetilde{y},\theta} \left(i\varphi(x,\widetilde{y},\theta) + 2 \Psi_2(\widetilde{y},z)\right).
\endeq

\bigskip
\noindent
Let us now make the following basic observation.
\begin{prop}
\label{prop1}
The holomorphic quadratic form $\Psi(x,z)$ given in {\rm (\ref{eq3.10})} satisfies
\begeq
\label{eq3.11}
2{\rm Re}\, \Psi(x,\overline{y}) \leq \Phi(x) + \Phi_2(y),\quad (x,y)\in \comp^n_x \times \comp^n_y.
\endeq
\end{prop}
\begin{proof}
It will be more convenient to verify that
\begeq
\label{eq3.11.1}
2{\rm Re}\, \Psi(x,y) \leq \Phi(x) + \Phi^*_2(y),\quad (x,y)\in \comp^n_x \times \comp^n_y,
\endeq
where $\Phi_2^*(y) = \Phi_2(\overline{y})$. A direct calculation shows that
$$
\frac{2}{i}\partial _y\Phi _2^*(y)=-\overline{\frac{2}{i}(\partial_y\Phi _2)(\overline{y}), }
$$
or equivalently,
$$
\frac{2}{i}\partial _y(\Phi _2^*)(\overline{y})=-\overline{\frac{2}{i}(\partial _y\Phi _2)(y). }
$$
It follows that the antilinear involution
\begeq
\label{eq3.11.2}
\Gamma: \comp^{2n} \ni (y,\eta) \mapsto (\overline{y},-\overline{\eta}) \in \comp^{2n}
\endeq
maps $\Lambda_{\Phi_2}$ bijectively onto $\Lambda_{\Phi_2^*}$. We conclude in view of (\ref{eq3.1}) that
\begeq
\label{eq3.11.3}
\kappa\circ \Gamma: \Lambda_{\Phi_2^*} \rightarrow \Lambda_{\Phi},
\endeq
and let us consider the graph of the map in (\ref{eq3.11.3}), ${\rm Graph}(\kappa \circ \Gamma) \cap \left(\Lambda_{\Phi} \times \Lambda_{\Phi_2^*}\right)$. Here $\Lambda_{\Phi} \times \Lambda_{\Phi_2^*} = \Lambda_{\Phi(x) + \Phi_2^*(y)}$ is I-Lagrangian and R-symplectic for the standard symplectic form
\begeq
\label{eq3.11.4}
d\xi \wedge dx + d\eta \wedge dy,
\endeq
on $\comp^{2n}_{x,\xi} \times \comp^{2n}_{y,\eta}$ and we claim that ${\rm Graph}(\kappa \circ \Gamma) \cap \left(\Lambda_{\Phi} \times \Lambda_{\Phi_2^*}\right)$ is Lagrangian for the symplectic form in (\ref{eq3.11.4}), restricted to $\Lambda_{\Phi} \times \Lambda_{\Phi_2^*}$. This can be seen by a direct computation: when $(t,s)\in \Lambda_{\Phi_2^*} \times \Lambda_{\Phi_2^*}$ we have, writing $\sigma$ for the standard symplectic form on $\comp^{2n}$,
$$
\sigma(\kappa(\Gamma(t)), \kappa(\Gamma(s))) + \sigma(t,s) = \sigma(\Gamma(t), \Gamma(s)) + \sigma(t,s) = - \overline{\sigma(t,s)} + \sigma(t,s) = 0,
$$
since $\sigma(t,s)$ is real. Here we have also used that, by a straightforward computation,
\begeq
\label{eq3.11.5}
\sigma (\Gamma t,\Gamma s) = -\overline{\sigma (t,s)}.
\endeq
It is then well known that $\pi_{x,y}\left({\rm Graph}(\kappa \circ \Gamma)\cap \left(\Lambda_{\Phi} \times \Lambda_{\Phi_2^*}\right)\right)$, the projection of ${\rm Graph}(\kappa \circ \Gamma)\cap (\Lambda_{\Phi} \times \Lambda_{\Phi_2^*})$ in $\comp^{2n}_{x,y}$, is maximally totally real, see~\cite{MeSj2}.

\medskip
\noindent
We now come to check (\ref{eq3.11.1}). To this end, we observe that (\ref{eq3.10}) gives
\begeq
\label{eq3.12}
2 \partial_x \Psi(x,y) = i \partial_x \varphi(x,\widetilde{y},\theta)
\endeq
and
\begeq
\label{eq3.13}
2 \partial_y \Psi (x,y) = 2 \partial_y \Psi_2 (\widetilde{y},y),
\endeq
where
\begeq
\label{eq3.14}
\partial_{\theta} \varphi(x,\widetilde{y},\theta) = 0, \quad \partial_{\widetilde{y}} \varphi(x,\widetilde{y},\theta) + \frac{2}{i} \partial_{\widetilde{y}}\Psi_2(\widetilde{y},y) = 0.
\endeq
We shall consider (\ref{eq3.12}), (\ref{eq3.13}) at the points $(x,y) \in \pi_{x,y}\left({\rm Graph}(\kappa \circ \Gamma)\cap \Lambda_{\Phi}\times \Lambda_{\Phi_2^*}\right)$, which corresponds to $\widetilde{y} = \overline{y}$ in (\ref{eq3.14}). Using (\ref{eq3.14}) together with the fact that
$$
\partial_{\widetilde{y}}\Psi_2(\widetilde{y}, \overline{\widetilde{y}}) = \partial_{\widetilde{y}} \Phi_2(\widetilde{y}),
$$
and (\ref{eq3.13}) together with the fact that
$$
\left(\partial_y \Psi_2\right) (\overline{y},y) = \partial_y \Phi_2^*(y),
$$
we conclude that at the points
$$
(x,y)\in \pi_{x,y}\left({\rm Graph}(\kappa \circ \Gamma)\cap \Lambda_{\Phi}\times \Lambda_{\Phi_2^*}\right),
$$
the following equalities hold,
\begeq
\label{eq3.15}
\partial_x \Psi(x,y) = \partial_x \Phi(x),\quad \partial_y \Psi (x,y) = \partial_y \Phi_2^*(y).
\endeq
In other words,
$$
\partial_x \left(\Phi(x) - 2{\rm Re}\, \Psi(x,y)\right) = \partial_y \left(\Phi_2^*(y) - 2{\rm Re}\, \Psi(x,y)\right) =0,
$$
along $\pi_{x,y}\left({\rm Graph}(\kappa \circ \Gamma)\cap \Lambda_{\Phi}\times \Lambda_{\Phi_2^*}\right)$, and thus the gradient of the real valued function
\begeq
\label{eq3.16}
F(x,y) = \Phi(x) + \Phi_2^*(y) - 2{\rm Re}\, \Psi(x,y)
\endeq
vanishes on $\pi_{x,y}\left({\rm Graph}(\kappa \circ \Gamma)\cap \Lambda_{\Phi}\times \Lambda_{\Phi_2^*}\right)$. It follows that the strictly plurisubharmonic quadratic form $F(x,y)$ vanishes to the second order along
\begeq
\label{eq3.16.1}
\pi_{x,y}\left({\rm Graph}(\kappa \circ \Gamma)\cap \Lambda_{\Phi}\times \Lambda_{\Phi_2^*}\right),
\endeq
and since the latter is maximally totally real, we get $F\geq 0$, thus implying (\ref{eq3.11.1}).
\end{proof}

\medskip
\noindent
{\it Remark}. The strictly plurisubharmonic quadratic form $F(x,y)$ in (\ref{eq3.16}) vanishes to the second order along the maximally totally real subspace (\ref{eq3.16.1}), and therefore the conclusion that $F\geq 0$ can be strengthened to the following,
$$
F(x,y) \asymp {\rm dist}\left((x,y), \pi_{x,y}\left({\rm Graph}(\kappa \circ \Gamma)\cap \Lambda_{\Phi}\times \Lambda_{\Phi_2^*}\right)\right)^2.
$$

\bigskip
\noindent
Let us now return to the Bergman type representation of the Fourier integral operator $A$ in (\ref{eq3.3}) quantizing $\kappa$. Combining (\ref{eq3.5}) and (\ref{eq3.9}), we get
\begin{equation}
\label{eq3.17}
Au(x)=\iint \check{a} e^{2(\Psi (x,\overline{y})-\Phi _2(y))}u(y) dyd\overline{y},
\end{equation}
for some $\check{a}\in {\bf C}$. This can be viewed as a Fourier integral operator
\begin{equation}
\label{eq3.18}
Au(x) = \iint \check{a} e^{2(\Psi (x,\theta )-\Psi _2(y,\theta ))}u(y) dy\,d\theta,
\end{equation}
where we take the integration contour $\theta =\overline{y}$ in (\ref{eq3.18}).

\medskip
\noindent
Since $\partial _y\partial _\theta \Psi _2(y,\theta )$ is non-degenerate, the phase function
\begin{equation}
\label{eq3.19}
\phi (x,y,\theta )=\frac{2}{i}(\Psi (x,\theta )-\Psi _2(y,\theta ))
\end{equation}
is non-degenerate in the sense of H\"ormander, and the canonical transformation $\kappa$ takes the form
\begin{equation}
\label{eq3.20}
\kappa:\ \left(y , \frac{2}{i}\partial _y\Psi _2(y,\theta ) \right)
\mapsto \left(x,\frac{2}{i}\partial _x\Psi (x,\theta ) \right),\hbox{ with
}\partial _\theta \Psi (x,\theta )=\partial _\theta \Psi _2(y,\theta ).
\end{equation}
We may also notice here that if we define
$$
\kappa _\Psi :\ (\theta , -(2/i)\partial _\theta \Psi (y,\theta))\mapsto (y, (2/i)\partial _y\Psi (y,\theta ))
$$
and $\kappa _{\Psi _2}$ similarly, then $\kappa = \kappa _\Psi \circ \kappa _{\Psi _2}^{-1}$.

\medskip
\noindent
The discussion so far shows that the canonical transformation $\kappa$ enjoying the mapping properties (\ref{eq3.1}), (\ref{eq3.1.1}), admits a non-degenerate phase function of the form (\ref{eq3.19}), where the quadratic form $\Psi$ satisfies
\begeq
\label{eq3.21}
2{\rm Re}\, \Psi(x,\overline{y}) \leq \Phi_1(x) + \Phi_2(y),\quad (x,y)\in \comp^n_x \times \comp^n_y.
\endeq
The positivity of $\kappa$ relative to $(\Lambda_{\Phi_1}, \Lambda_{\Phi_2})$ is then implied by the following general result.

\medskip
\noindent
\begin{prop}
\label{prop2}
Let $\kappa$ be a canonical transformation satisfying {\rm (\ref{eq3.1})} and let us consider a metaplectic Fourier integral operator of the form {\rm (\ref{eq3.17})}, or equivalently {\rm (\ref{eq3.18})}, associated to $\kappa$. Then the following conditions are equivalent:
\begin{itemize}
\item[(i)] $\kappa$ is positive relative to $(\Lambda _{\Phi _1},\Lambda _{\Phi _2})$ in the sense of {\rm (\ref{eq1.4})},
\begin{equation}
\label{eq3.22}
\frac{1}{i}\sigma (t_1,\iota _{\Phi _1}t_1)-\frac{1}{i}\sigma
(t_2,\iota _{\Phi _2}t_2)\ge 0,\hbox{ whenever }t_1 = \kappa(t_2),\quad t_2 \in \comp^{2n}.
\end{equation}
\item[(ii)] $\Lambda _{2{\rm Re}\, \Psi (x,\overline{y})}$ is positive relative to $\Lambda _{\Phi _1(x)+\Phi _2(y)}$,
\item[(iii)] $2\Re \Psi (x,\overline{y})-\Phi _1(x)-\Phi _2(y)\le 0$
  on ${\bf C}_x^n\times {\bf C}_y^n$.
\end{itemize}
\end{prop}
\begin{proof}
The equivalence (ii)$\Leftrightarrow$(iii) follows from Theorem \ref{theo_first}, so it suffices to show the equivalence (i)$\Leftrightarrow$(ii).

\medskip
\noindent
Clearly, (iii) is equivalent to
\begin{equation}
\label{eq3.23}
2\Re \Psi (x,y)-\Phi _1(x)-\Phi _2^*(y)\le 0 \hbox{ on }{\bf C}^{2n}_{x,y},
\end{equation}
where $\Phi _2^*(y)=\Phi _2(\overline{y})(=\overline{\Phi_2(\overline{y})})$, and by Theorem \ref{theo_first} (ii) is equivalent to
\begin{equation}
\label{eq3.24}
\Lambda _{2 {\rm Re}\, \Psi (x,y)} \hbox{ is positive relative to }\Lambda _{\Phi_1(x)+\Phi _2^*(y)}.
\end{equation}
We have,
\begin{equation}
\begin{split}
\Lambda _{2 {\rm Re}\, \Psi }&=\{ (x,(2/i) \partial _x 2\Re \Psi (x,y);
y,(2/i) \partial _y 2\Re \Psi (x,y) \}\\
&=\{ (x,(2/i) \partial _x \Psi (x,y);
y,(2/i) \partial _y \Psi (x,y) \},
\end{split}
\end{equation}
and (\ref{eq3.24}) means that
\begin{equation}
\label{eq3.25}
\frac{1}{i}\sigma (t_1,\iota _{\Phi _1}t_1)+\frac{1}{i}\sigma (t_2,\iota _{\Phi _2^*}t_2)\ge 0, \hbox{ for all }(t_1,t_2)\in \Lambda
_{2 {\rm Re}\, \Psi }.
\end{equation}

\medskip
\noindent
Here, we shall relate the involutions $\iota _{\Phi _2^*}$ and $\iota _{\Phi_2}$. From (\ref{eq2.3.1}) let us recall that $\iota _{\Phi _2}$ is given by
\begin{equation}
\label{eq3.26}
\iota _{\Phi _2}:\ \left(y,(2/i)\overline{\partial _y\Psi
  _2(x,\overline{y})}\right) \mapsto \left(x,(2/i)\partial _x\Psi _2(x,\overline{y})\right).
\end{equation}
We also know that the antilinear involution $\Gamma$, given in (\ref{eq3.11.2}), maps $\Lambda_{\Phi_2}$ bijectively onto $\Lambda_{\Phi_2^*}$, and since $\iota _{\Phi _2}$, $\iota _{\Phi _2^*}$ are the unique antilinear maps equal to the identity on $\Lambda _{\Phi _2}$ and
$\Lambda _{\Phi _2^*}$ respectively, it follows that
\begin{equation}
\label{eq3.27}
\iota _{\Phi _2^*}=\Gamma \iota _{\Phi _2}\Gamma .
\end{equation}

\medskip
\noindent
From (\ref{eq3.11.5}), let us recall that 
$$
\frac{1}{i}\sigma (\Gamma t,\Gamma s)=\overline{\frac{1}{i}\sigma (t,s)},
$$
so using (\ref{eq3.27}), we find that the second term in (\ref{eq3.25}) is equal to
$$
\frac{1}{i}\sigma (t_2,\Gamma \iota _{\Phi _2}\Gamma t_2)=\overline{\frac{1}{i}\sigma (\Gamma t_2,\iota _{\Phi _2}\Gamma
    t_2)}=\frac{1}{i}\sigma (\Gamma t_2,\iota _{\Phi _2}\Gamma
  t_2)=-\frac{1}{i}\sigma (\iota _{\Phi _2}\Gamma t_2,\Gamma t_2),
$$
where we also used the fact that $(1/i)\sigma (t,\iota _{\Phi _2}t)$ is real. Hence (\ref{eq3.24}) is equivalent, via (\ref{eq3.25}), to
\begin{equation}
\label{eq3.28}
\frac{1}{i}\sigma (t_1,\iota _{\Phi _1}t_1)-\frac{1}{i}\sigma (\iota
_{\Phi _2}\Gamma t_2 ,\Gamma t_2 )\ge 0,\ \forall (t_1,t_2)\in \Lambda _{2 {\rm Re}\, \Psi }.
\end{equation}

\medskip
\noindent
From (\ref{eq3.26}), we get
$$
\iota _{\Phi _2}\Gamma :\
\left(\overline{y},-\overline{\frac{2}{i}\overline{\partial _y\Psi
      _2(x,\overline{y})}} \right)
\mapsto
\left(x,\frac{2}{i}\partial _x\Psi _2(x,\overline{y}) \right),
$$
i.e.
\begin{equation}
\label{eq3.29}
\iota _{\Phi _2}\Gamma :\ \left(\theta ,\frac{2}{i}\partial _\theta
  \Psi _2(y,\theta ) \right)\mapsto \left(y,\frac{2}{i}\partial _y\Psi
_2(y,\theta )\right),
\end{equation}
where we changed the notation slightly for convenience.

\medskip
\noindent
Write,
$$
\Lambda _{2 {\rm Re}\, \Psi }\ni (t_1,t_2)=\left(x,\frac{2}{i}\partial _x\Psi (x,\theta
);\theta ,\frac{2}{i}\partial _\theta \Psi (x,\theta )\right),
$$
and put $t_3=\iota _{\Phi _2}\Gamma t_2$, so that by (\ref{eq3.29}),
$$
t_3=\left( y,\frac{2}{i}\partial _y\Psi _2 (y,\theta ) \right) ,
$$
where
$$
\left( \theta ,\frac{2}{i}\partial _\theta \Psi (x,\theta ) \right)=
\left( \theta ,\frac{2}{i}\partial _\theta \Psi_2 (y,\theta ) \right).
$$
Comparing with (\ref{eq3.20}), we see that $t_1 = \kappa (t_3)$. Since
$\Gamma t_2=\iota _{\Phi _2}^2 \Gamma t_2=\iota _{\Phi _2}t_3$, we see
that (\ref{eq3.28}) is equivalent to
\begin{equation}
\label{eq3.30}
\frac{1}{i}\sigma (t_1,\iota _{\Phi _1}t_1)-\frac{1}{i}\sigma
(t_3,\iota _{\Phi _2}t_3)\ge 0,\hbox{ when }t_1 = \kappa(t_3),
\end{equation}
which is precisely (\ref{eq3.22}) up to a change of notation. This completes the proof of the equivalence (i)$\Leftrightarrow$(ii) and of
the proposition.
\end{proof}

\bigskip
\noindent
Combining Proposition \ref{prop1} and Proposition \ref{prop2}, we see that the proof of the sufficiency part of Theorem \ref{theo_main} is now complete.

\bigskip
\noindent
{\it Remark}: Let $\kappa: \comp^{2n} \rightarrow \comp^{2n}$ be a complex linear canonical transformation, such that (\ref{eq3.1}) holds, where $\Phi_2$, $\Phi$ are strictly plurisubharmonic. It follows from (\ref{eq3.15}) that the holomorphic quadratic form $\Psi(x,y)$ depends only on $\kappa$ and on the weights $\Phi_2$, $\Phi$, but not on the choice of a non-degenerate phase function $\varphi(x,y,\theta)$, $(x,y,\theta) \in \comp^n_x \times \comp^n_y \times \comp^N_{\theta}$, such that
$$
\Lambda_{\varphi}' = {\rm Graph}(\kappa),
$$
where
$$
\Lambda_{\varphi}' = \{(x,\varphi'_x(x,y,\theta); y, -\varphi'_y(x,y,\theta));\,\, \varphi'_{\theta}(x,y,\theta) = 0\}.
$$
It follows that if $\psi(x,y,w)$, $(x,y,w) \in \comp^n_x \times \comp^n_y \times \comp^{N'}_{w}$, is a second non-degenerate phase function such that
$$
\Lambda_{\varphi}' = \Lambda_{\psi}' = {\rm Graph}(\kappa),
$$
then both $\varphi$ and $\psi$ give rise to the same Fourier integral operators, realized as bounded linear maps: $H_{\Phi_2}(\comp^n) \rightarrow H_{\Phi}(\comp^n)$.

\bigskip
\noindent
We shall finish this section by making some remarks concerning metaplectic Fourier integral operators in the complex domain, associated to canonical transformations that are strictly positive relative to $(\Lambda_{\Phi_1}, \Lambda_{\Phi_2})$. Let
\begeq
\label{eq3.31}
\kappa: \comp^{2n} \rightarrow \comp^{2n}
\endeq
be a complex linear canonical transformation which is strictly positive relative to $(\Lambda_{\Phi_1}, \Lambda_{\Phi_2})$. According to Theorem \ref{theo_main}, we then have
\begeq
\label{eq3.32}
\kappa(\Lambda_{\Phi_2}) = \Lambda_{\Phi},
\endeq
where $\Phi$ is a strictly plurisubharmonic quadratic form on $\comp^n$ such that
\begeq
\label{eq3.34}
\Phi_1(x) - \Phi(x) \asymp \abs{x}^2,\quad x\in \comp^n.
\endeq
Let
$$
Tu(x) = \int\!\!\! \int e^{i \phi(x,y,\theta)} a u(y)\,dy\, d\theta,\quad a\in \comp,
$$
be a Fourier integral operator associated to $\kappa$. As discussed above, it follows from~\cite{CGHS},~\cite{Sj82} that the operator $T$ can be realized by means of a suitable good contour and we then obtain a bounded operator
\begeq
\label{eq3.35}
T: H_{\Phi_2}(\comp^n) \rightarrow H_{\Phi}(\comp^n).
\endeq
It follows from (\ref{eq3.34}) that the inclusion map $H_{\Phi}(\comp^n) \rightarrow H_{\Phi_1}(\comp)$ is compact, and the operator $T: H_{\Phi_2}(\comp^n) \rightarrow H_{\Phi_1}(\comp^n)$ is therefore compact. The following sharpening is essentially well known, see~\cite{AV}.

\begin{prop}
\label{prop_trace}
The operator
$$
T: H_{\Phi_2}(\comp^n) \rightarrow H_{\Phi_1}(\comp^n)
$$
is of trace class, with the singular values $s_j(T)$ satisfying
\begeq
\label{eq3.36}
s_j(T) = {\cal O}(j^{-\infty}).
\endeq
\end{prop}
\begin{proof}
Let $q$ be a holomorphic quadratic form on $\comp^{2n}$ such that its restriction to $\Lambda_{\Phi_1}$ is real positive definite. Let us introduce the Weyl quantization of $q$, the operator $Q = q^w(x,D_x)$. The quadratic differential operator $Q$ is selfadjoint on $H_{\Phi_1}(\comp^n)$ with discrete spectrum, and let us consider the metaplectic Fourier integral operator $e^{t Q}$, $0\leq t \leq t_0 \ll 1$, acting on the space $H_{\Phi}(\comp^n)$. Using some well known arguments, explained in detail in~\cite{HeSjSt},~\cite{HiPr},~\cite{HiPrVi}, we see that for $t\in [0,t_0]$ with $t_0 > 0$ small enough, the operator $e^{tQ}$ is bounded,
\begeq
\label{eq3.37}
e^{tQ}: H_{\Phi}(\comp^n) \rightarrow H_{\Phi_t}(\comp^n),
\endeq
where $\Phi_t$ is a strictly plurisubharmonic quadratic form on $\comp^n$, depending smoothly on $t\geq 0$ small enough, such that
\begeq
\label{eq3.38}
\Phi_t(x) = \Phi(x) + {\cal O}(t)\abs{x}^2.
\endeq
Combining this observation with (\ref{eq3.34}) we conclude that there exists $\delta >0$ small enough such  that the operator
\begeq
\label{eq3.39}
e^{\delta Q} T: H_{\Phi_2}(\comp^n) \rightarrow H_{\Phi_1}(\comp^n)
\endeq
is bounded. Writing
\begeq
\label{eq3.40}
T = e^{-\delta Q} e^{\delta Q} T,
\endeq
and applying the Ky Fan inequalities, we get
$$
s_j(T) \leq s_j(e^{-\delta Q}) \norm{e^{\delta Q} T}_{{\cal L}(H_{\Phi_2},H_{\Phi_1})} = {\cal O}(j^{-\infty}).
$$
Here we have also used the fact that the singular values of the compact positive selfadjoint operator $e^{-\delta Q}$ on $H_{\Phi_1}(\comp^n)$ satisfy
$$
s_j(e^{-\delta Q}) = {\cal O}(j^{-\infty}).
$$
It follows that $T$ is of trace class and the proof of the proposition is complete.
\end{proof}

\section{Applications to Toeplitz operators}
\label{sect_four}
\setcounter{equation}{0}
The purpose of this section is to apply the point of view of Fourier integral operators in the complex domain, developed in the previous sections, to the study of Toeplitz operators in the Bargmann space, establishing Theorem \ref{T_bound}.

\medskip
\noindent
Let $\Phi_0$ be a strictly plurisubharmonic quadratic form on $\comp^n$ and let $p: \comp^n \rightarrow \comp$ be measurable. Associated to $p$ is the Toeplitz operator
\begeq
\label{eq4.1}
{\rm Top}(p) = \Pi_{\Phi_0} \circ p \circ \Pi_{\Phi_0}: H_{\Phi_0}(\comp^n) \rightarrow H_{\Phi_0}(\comp^n).
\endeq
Here
$$
\Pi_{\Phi_0}: L^2(\comp^n,e^{-2\Phi_0} L(dx)) \rightarrow H_{\Phi_0}(\comp^n)
$$
is the orthogonal projection. We shall always assume that when equipped with the natural domain
\begeq
\label{eq4.2}
{\cal D}({\rm Top}(p)) = \{u\in H_{\Phi_0}(\comp^n); p u\in L^2(\comp^n, e^{-2\Phi_0} L(dx))\},
\endeq
the operator ${\rm Top}(p)$ becomes densely defined.

\medskip
\noindent
For future reference, let us recall the link between the Toeplitz and Weyl quantizations on $\comp^n$. Let $p\in L^{\infty}(\comp^n)$, say. Then we have
\begeq
\label{eq4.3}
{\rm Top}(p) = a^w(x,D_x),
\endeq
where $a\in C^{\infty}(\Lambda_{\Phi_0})$ is given by
\begeq
\label{eq4.4}
a\left(x,\frac{2}{i}\frac{\partial \Phi_0}{\partial x}(x)\right)  = \left(\exp\left(\frac{1}{4} \left(\Phi''_{0,x\overline{x}}\right)^{-1} \partial_x \cdot \partial_{\overline{x}}\right)p\right)(x), \quad x\in \comp^n.
\endeq
See~\cite{G84},~\cite{Sj96}. Here $-\left(\Phi''_{0,x\overline{x}}\right)^{-1} \partial_x \cdot \partial_{\overline{x}}$ is a constant coefficient second order differential operator on $\comp^n$ whose symbol is the positive definite quadratic form
$$
\frac{1}{4} \left(\Phi''_{0,x \overline{x}}\right)^{-1} \overline{\xi} \cdot \xi>0, \quad 0 \neq \xi \in \comp^n \simeq \real^{2n},
$$
and therefore the operator in (\ref{eq4.4}) can be regarded as the forward heat flow acting on $p$.

\medskip
\noindent
In this section we shall be concerned with the question of when an operator of the form ${\rm Top}(p)$ is bounded,
$$
{\rm Top}(p) \in {\cal L}(H_{\Phi_0}(\comp^n), H_{\Phi_0}(\comp^n)),
$$
and following~\cite{BC94}, in doing so we shall only consider Toeplitz symbols of the form
\begeq
\label{eq4.5}
p = e^{2q},
\endeq
where $q$ is a complex-valued quadratic form on $\comp^n$. Let us first proceed to give an explicit criterion, guaranteeing that when equipped with the domain (\ref{eq4.2}), the operator ${\rm Top}(e^{2q})$ is densely defined. Recalling the decomposition (\ref{eq2.3.5}) and considering the unitary map
$$
H_{\Phi_0}(\comp^n) \ni u \mapsto u e^{-f}\in H_{\Phi_{\rm herm}}(\comp^n), \quad f(x) = \Phi_{0,xx}''x\cdot x,
$$
we may observe that the space $e^{f}{\cal P}(\comp^n) = \{e^f p; p\in {\cal P}(\comp^n)\}$ is dense in $H_{\Phi_0}(\comp^n)$. Here ${\cal P}(\comp^n)$ is the space of holomorphic polynomials on $\comp^n$. It follows that
$$
e^{f}{\cal P}(\comp^n) \subset {\cal D}({\rm Top}(e^{2q})),
$$
so that ${\rm Top}(e^{2q})$ is densely defined, provided that
\begeq
\label{eq4.6}
2 {\rm Re}\, q(x) < \Phi_{{\rm herm}}(x),
\endeq
in the sense of quadratic forms on $\comp^n$.

\bigskip
\noindent
Recalling (\ref{eq3.52}), we may write
\begeq
\label{eq4.7}
{\rm Top}(e^{2q}) u(x) = C \int e^{2\left(\Psi_0(x,\overline{y}) - \Phi_0(y)\right)} e^{2 q(y,\overline{y})} u(y)\, dy\, d\overline{y}, \quad u \in {\cal D}({\rm Top}(e^{2q})).
\endeq
Here $C>0$ and $\Psi_0$ is the polarization of $\Phi_0$. Similarly to (\ref{eq3.18}), we get
\begeq
\label{eq4.8}
{\rm Top}(e^{2q})u(x) = C \int\!\!\!\int_{\Gamma} e^{2(\Psi_0(x,\theta) - \Psi_0(y,\theta) + q(y,\theta))} u(y)\, dy\, d\theta,
\endeq
where $\Gamma$ is the contour in $\comp^{2n}$, given by $\theta = \overline{y}$. Here the holomorphic quadratic form
\begeq
\label{eq4.9}
F(x,y,\theta) = \frac{2}{i}\left(\Psi_0(x,\theta) - \Psi_0(y,\theta) + q(y,\theta)\right)
\endeq
is a non-degenerate phase function in the sense of H\"ormander, in view of the fact that ${\rm det}\, \Psi''_{0,x\theta} \neq 0$, and therefore the operator ${\rm Top}(e^{2q})$ in (\ref{eq4.8}) can be viewed as a metaplectic Fourier integral operator associated to a suitable canonical relation
$\subset \comp^{2n} \times \comp^{2n}$. We have the formal factorization
$$
{\rm Top}(e^{2q}) = AB,
$$
where
\begeq
\label{eq4.10}
A v(x) = \int e^{2 \Psi_0(x,\theta)} v(\theta)\, d\theta,\quad  B u(\theta) = \int e^{-2\widetilde{\Psi}_0(y,\theta)} u(y)\,dy,
\endeq
and where we have written $\widetilde{\Psi}_0(y,\theta) = \Psi_0(y,\theta) - q(y,\theta)$. Here the operator $A$, formally, is an elliptic Fourier integral operator associated to the canonical transformation
$$
(\theta, - \frac{2}{i} \partial_{\theta} \Psi_0(x,\theta)) \mapsto (x, \frac{2}{i}\partial_x \Psi_0(x,\theta)).
$$
It follows that the canonical relation associated to ${\rm Top}(e^{2q})$ is the graph of a canonical transformation if and only if this is the case for the Fourier integral operator $B$. We conclude that the operator ${\rm Top}(e^{2q})$ in (\ref{eq4.8}) is associated to a canonical transformation precisely when
\begeq
\label{eq4.11}
\partial_y \partial_{\theta} \widetilde{\Psi}_0 \neq 0.
\endeq
The condition (\ref{eq4.11}) is equivalent to the assumption (\ref{eq1.8}) in Theorem \ref{T_bound}. The canonical transformation is then given by
\begeq
\label{eq4.11.1}
\kappa: (y, - \partial_{y} F(x,y,\theta)) \mapsto (x, \partial_x F(x,y,\theta)),\quad \partial_{\theta} F(x,y,\theta) = 0.
\endeq

\bigskip
\noindent
{\it Example}. In the following discussion, we shall revisit the family of examples discussed in Section 6 of~\cite{BC94} and show how the point of view of Fourier integral operators in the complex domain, developed above, allows one to recover the main findings of Section 6 in~\cite{BC94}, obtained there by means of a direct computation.

\medskip
\noindent
Let $\Phi_0(x) = (1/2) \abs{x}^2$ and $q = (\lambda/2) \abs{y}^2$, $\lambda \in \comp$ with ${\rm Re}\, \lambda < 1/2$. Here the restriction on
${\rm Re}\, \lambda$ implies that (\ref{eq4.6}) holds, so that the operator ${\rm Top}(e^{2q})$ is densely defined in $H_{\Phi_0}(\comp^n)$. We have
$$
\Psi_0(x,y) = \frac{1}{2} x \cdot y,
$$
and the phase function $F$ in (\ref{eq4.9}) is given by
\begeq
\label{eq4.12}
F(x,y,\theta) = \frac{2}{i}\left(\frac{1}{2} x\cdot \theta - \left(\frac{1-\lambda}{2}\right)y\cdot \theta\right).
\endeq
In particular, the condition (\ref{eq4.11}) is satisfied and we may then compute the canonical transformation $\kappa$ associated to the corresponding Fourier integral operator ${\rm Top}(e^{2q})$ in (\ref{eq4.8}).

\medskip
\noindent
The critical set $C_F$ of the phase $F$ is given by $\partial_{\theta} F = 0 \Longleftrightarrow x = (1-\lambda)y$, and the corresponding canonical transformation $\kappa$ is of the form
\begeq
\label{eq4.13}
\kappa: (y, - \partial_{y} F(x,y,\theta)) \mapsto (x, \partial_x F(x,y,\theta)),\quad (x,y,\theta) \in C_F.
\endeq
It follows that $\kappa$ is given by
\begeq
\label{eq4.14}
\kappa: (y, \eta) \mapsto \left((1-\lambda)y, \frac{\eta}{1-\lambda}\right).
\endeq
We shall now determine when the canonical transformation $\kappa$ is positive relative to $\Lambda_{\Phi_0}$, which can be done by a direct computation: it follows from (\ref{eq2.3.1}) that the involution $\iota_{\Phi_0}$ is given by
\begeq
\label{eq4.15}
\iota_{\Phi_0}: (y,\eta) \mapsto \left(\frac{1}{i}\overline{\eta}, \frac{1}{i} \overline{y}\right),
\endeq
and therefore, we may compute,
\begin{multline}
\label{eq4.16}
\frac{1}{i} \sigma(\kappa(y,\eta), \iota_{\Phi_0} \kappa(y,\eta)) = \frac{1}{i} \sigma \left(((1-\lambda)y, \frac{\eta}{1-\lambda}), (\frac{1}{i} \frac{\overline{\eta}}{1-\overline{\lambda}}, \frac{1}{i}(1-\overline{\lambda}) \overline{y})\right) \\
= \abs{1-\lambda}^2 \abs{y}^2 - \frac{\abs{\eta}^2}{\abs{1-\lambda}^2}.
\end{multline}
Similarly, we have
\begeq
\label{eq4.17}
\frac{1}{i} \sigma((y,\eta),\iota_{\Phi_0}(y,\eta)) = \abs{y}^2 - \abs{\eta}^2.
\endeq
Combining (\ref{eq4.16}), (\ref{eq4.17}) we see that the $\kappa$ is positive relative to $\Lambda_{\Phi_0}$ if and only if
\begeq
\label{eq4.18}
\abs{1-\lambda} \geq 1.
\endeq
This condition occurs in~\cite{BC94}, pp. 581, 582 (with the inessential difference that in the discussion in~\cite{BC94} one considers $\Phi_0(x) = \abs{x}^2/4$), where it is verified that the operator ${\rm Top}(e^{2q})\in {\cal L}(H_{\Phi_0}(\comp^n), H_{\Phi_0}(\comp^n))$ precisely when (\ref{eq4.18}) holds.

\medskip
\noindent
In the case when the strict inequality holds in (\ref{eq4.18}), the canonical transformation $\kappa$ in (\ref{eq4.14}) is strictly positive relative to $\Lambda_{\Phi_0}$ and it follows from Proposition \ref{prop_trace} that the Toeplitz operator ${\rm Top}(e^{2q})$ is of trace class on $H_{\Phi_0}(\comp^n)$.

\medskip
\noindent
We shall now proceed to discuss the "boundary" case when
\begeq
\label{eq4.19}
\abs{1-\lambda} = 1.
\endeq
In this case, using (\ref{eq4.14}) we immediately see that $\kappa(\Lambda_{\Phi_0}) = \Lambda_{\Phi_0}$, and therefore we conclude, in view of~\cite{CGHS},~\cite{Sj82}, that the operator
\begeq
\label{eq4.20}
{\rm Top}(e^{2q}): H_{\Phi_0}(\comp^n) \rightarrow H_{\Phi_0}(\comp^n)
\endeq
is bounded, with a bounded two-sided inverse.

\medskip
\noindent
We claim next that the operator in (\ref{eq4.20}) is in fact unitary when (\ref{eq4.19}) holds, and when verifying the unitarity, it will be convenient to pass to the Weyl quantization, computing the Weyl symbol of ${\rm Top}(e^{2q})$. It follows from (\ref{eq4.4}) that
\begeq
\label{eq4.21}
a\left(x,\frac{2}{i}\frac{\partial \Phi_0}{\partial x}(x)\right) = \left(\exp\left(\frac{1}{8}\Delta\right)e^{2q}\right)(x) =
\left(\frac{2}{\pi}\right)^n \int_{{\bf C}^n} e^{-2\abs{x-y}^2} e^{\lambda\abs{y}^2}\,L(dy).
\endeq
Here $\Delta$ is the Laplacian on $\comp^n \simeq \real^{2n}$. Computing the Gaussian integral in (\ref{eq4.21}) by the exact version of stationary phase, we get, see also~\cite{BC94},
\begeq
\label{eq4.22}
a\left(x,\frac{2}{i}\frac{\partial \Phi_0}{\partial x}(x)\right) = \left(\frac{2}{2-\lambda}\right)^n \exp\left(\frac{2\lambda}{2-\lambda} \abs{x}^2\right).
\endeq
Here we may notice that
$$
{\rm Re}\left(\frac{2\lambda}{2-\lambda}\right)= 0,
$$
when (\ref{eq4.19}) holds, reflecting the fact that the associated canonical transformation in (\ref{eq4.14}) is "real" in this case. We conclude that the Weyl symbol of the Toeplitz operator ${\rm Top}(e^{2q})$ is given by
\begeq
\label{eq4.23}
a(x,\xi) = \left(\frac{2}{2-\lambda}\right)^n \exp(iF(x,\xi)), \quad F(x,\xi) = \frac{2\lambda}{2-\lambda} x\cdot \xi,
\endeq
so that
\begeq
\label{eq4.24}
{\rm Top}(e^{2q}) = \left(\frac{2}{2-\lambda}\right)^n \left(\exp(iF)\right)^w.
\endeq
We have $({\rm Im}\, F)|_{\Lambda_{\Phi_0}} = 0$ and an application of Proposition 5.11 of~\cite{H95} together with the metaplectic invariance of the Weyl quantization allows us to conclude that the operator
\begeq
\label{eq4.25}
\sqrt{{\rm det}(I-{\cal F}/2)} \left(\exp(iF)\right)^w: H_{\Phi_0}(\comp^n) \rightarrow H_{\Phi_0}(\comp^n)
\endeq
is unitary. Here ${\cal F}$ is the Hamilton map of $F$, i.e. the matrix of the (linear) Hamilton field $H_F$, and it remains therefore to check that
\begeq
\label{eq4.26}
\sqrt{{\rm det}(I-{\cal F}/2)} = \left(\frac{2}{2-\lambda}\right)^n e^{i\theta},\quad \theta \in \real.
\endeq
To this end, we compute using (\ref{eq4.23}),
$$
\frac{1}{2}{\cal F} = \frac{\lambda}{2-\lambda} \begin{pmatrix}1 &0\\ 0 &-1\end{pmatrix},\quad I-\frac{1}{2}{\cal F} = \frac{2}{2-\lambda} \begin{pmatrix}1-\lambda &0\\ 0 &1\end{pmatrix},
$$
and (\ref{eq4.26}) follows, thanks to (\ref{eq4.19}). We conclude therefore that the Toeplitz operator ${\rm Top}(e^{2q})$ is unitary on $H_{\Phi_0}(\comp^n)$, when ${\rm Re}\, \lambda < 1/2$ and (\ref{eq4.19}) holds. The unitarity property has also been observed in~\cite{BC94}.

\medskip
\noindent
{\it Remark}. In the case when ${\rm Re}\, \lambda < 1/2$, $\abs{1 - \lambda} > 1$, we observed that the operator ${\rm Top}(e^{2q})$ is of trace class on $H_{\Phi_0}(\comp^n)$, and we get, using (\ref{eq4.23}) and the metaplectic invariance of the Weyl quantization,
$$
{\rm tr}\, {\rm Top}(e^{2q}) = \frac{1}{(2\pi)^n} \int\!\!\!\int_{\Lambda_{\Phi_0}} a \frac{(\sigma|_{\Lambda_{\Phi_0}})^n}{n!},
$$
where $a$ is given in (\ref{eq4.23}).

\bigskip
\noindent
We are now ready to discuss the proof of Theorem \ref{T_bound}. It follows from Theorem \ref{theo_main} and the discussion in this section that it suffices to check that the canonical transformation (\ref{eq4.11.1}) associated to the operator ${\rm Top}(e^{2q})$ is positive relative to $\Lambda_{\Phi_0}$. To this end, let us consider the Weyl symbol of ${\rm Top}(e^{2q})$, given by (\ref{eq4.4}),
\begeq
\label{4.27}
a(x,\xi) = \left(\exp\left(\frac{1}{4} \left(\Phi''_{0,x\overline{x}}\right)^{-1} \partial_x \cdot \partial_{\overline{x}}\right) e^{2q}\right)(x),\quad (x,\xi) \in \Lambda_{\Phi_0}.
\endeq
A simple computation of the inverse Fourier transform of a real Gaussian shows that
\begeq
\label{eqrem.10.1}
a(x,\xi) = C_{\Phi_0} \int_{{\bf C}^n} \exp(-4 \Phi_{{\rm herm}}(x-y)) e^{2q(y)}\, L(dy),\quad C_{\Phi_0}\neq 0.
\endeq
Here the convergence of the integral in (\ref{eqrem.10.1}) is guaranteed by (\ref{eq4.6}). In view of the exact version of stationary phase, it is therefore clear that
\begeq
\label{rem.11}
a(x,\xi) = C \exp(iF(x,\xi)), \quad (x,\xi) \in \Lambda_{\Phi_0},
\endeq
for some constant $C\neq 0$, where $F$ is a holomorphic quadratic form on $\comp^{2n}$. Proposition \ref{prop_Weyl} shows that the positivity of $\kappa$ in (\ref{eq4.11.1}) relative to $\Lambda_{\Phi_0}$ is equivalent to the fact that the Weyl symbol in (\ref{rem.11}) is such that ${\rm Im}\, F|_{\Lambda_{\Phi_0}} \geq 0 \Longleftrightarrow\exp(iF)\in L^{\infty}(\Lambda_{\Phi_0})$. The proof of Theorem \ref{T_bound} is complete.

\begin{appendix}
\section{Schwartz kernel theorem in the $H_{\Phi}$--setting}
\label{appA}
\setcounter{equation}{0}
In this appendix we shall make some elementary remarks concerning integral representations for linear continuous maps between weighted spaces of
holomorphic functions. Such observations are essentially well known, see for instance~\cite{P90}.

\medskip
\noindent
Let $\Omega_j \subset \comp^{n_j}$ be open, $j=1,2$, and let $\Phi_j \in C(\Omega_j; \real)$. We introduce the weighted spaces
\begeq
\label{eqS.1}
H_{\Phi_j}(\Omega_j) = {\rm Hol}(\Omega_j) \cap L^2(\Omega_j, e^{-2\Phi_j} L(dy_j)),\quad j =1,2,
\endeq
where $L(dy_j)$ is the Lebesgue measure on $\comp^{n_j}$. When viewed as closed subspaces of $L^2(\Omega_j, e^{-2\Phi_j} L(dy_j))$, the
spaces $H_{\Phi_j}(\Omega_j)$ are separable complex Hilbert spaces and the natural embeddings
$H_{\Phi_j}(\Omega_j) \rightarrow {\rm Hol}(\Omega_j)$ are continuous. Here the space ${\rm Hol}(\Omega_j)$ is equipped with its natural Fr\'echet space
topology of locally uniform convergence.
Let
\begeq
\label{eqS.2}
T: H_{\Phi_1}(\Omega_1) \rightarrow H_{\Phi_2}(\Omega_2)
\endeq
be a linear continuous map. Let us also write $\overline{\Omega_1} = \{z\in \comp^{n_1};\,\, \overline{z}\in \Omega_1\}$.

\begin{theo}
\label{theo_kernel}
There exists a unique function $K(x,z)\in {\rm Hol}(\Omega_2 \times \overline{\Omega_1})$ such that
\begeq
\label{eqS.21}
\Omega_1 \ni y \mapsto \overline{K(x,\overline{y})} \in H_{\Phi_1}(\Omega_1),
\endeq
for each $x\in \Omega_2$, and
\begeq
\label{eqS.22}
Tf(x) = \int_{\Omega_1} K(x,\overline{y}) f(y) e^{-2\Phi_1(y)}\, L(dy),\quad f\in H_{\Phi_1}(\Omega_1).
\endeq
We also have
\begeq
\label{eqS.23}
\Omega_2 \ni x \mapsto K(x,z)\in H_{\Phi_2}(\Omega_2),
\endeq
for each $z\in \overline{\Omega_1}$.
\end{theo}

\medskip
\noindent
When proving Theorem \ref{theo_kernel}, we observe that it follows from the remarks above that for each $x\in \Omega_2$, the linear form
\begeq
\label{eqS.3}
H_{\Phi_1}(\Omega_1) \ni f \mapsto \left(Tf\right)(x) \in \comp
\endeq
is continuous, and there exists therefore a unique element $k_x\in H_{\Phi_1}(\Omega_1)$ such that for all $f\in H_{\Phi_1}(\Omega_1)$, we have
\begeq
\label{eqS.4}
Tf(x) = (f, k_x)_{\Phi_1},\quad x \in \Omega_2.
\endeq
Here and in what follows $(\cdot,\cdot)_{\Phi_j}$ stands for the scalar product in the space $H_{\Phi_j}(\Omega_j)$, $j=1,2$.

\medskip
\noindent
Letting $(e_j)$ be an orthonormal basis for $H_{\Phi_1}(\Omega_1)$, we may write with convergence in $H_{\Phi_1}(\Omega_1)$, for each $x\in \Omega_2$ fixed,
\begeq
\label{eqS.5}
k_x = \sum_{j=1}^{\infty} (k_x, e_j)_{\Phi_1} e_j  = \sum_{j=1}^{\infty} \overline{Te_j(x)} e_j.
\endeq
By Parseval's formula we get
\begeq
\label{eqS.6}
\norm{k_x}^2_{\Phi_1} = \sum_{j=1}^{\infty} \abs{Te_j(x)}^2, \quad x\in \Omega_2.
\endeq
Here we know that
\begeq
\label{eqS.7}
\norm{k_x}_{\Phi_1} = \sup_{\norm{f}_{\Phi_1} \leq 1} \abs{Tf(x)},
\endeq
and it follows that the function $\Omega_2 \ni x \mapsto \norm{k_x}_{\Phi_1}$ is locally bounded. Let us now make the following elementary observation: let $\Omega \subset \comp^n$ be open and let $f_n \in {\rm Hol}(\Omega)$ be such that the series
\begeq
\label{eqS.7.1}
\sum_{n=1}^{\infty} \abs{f_n(z)}^2
\endeq
converges for each $z\in \Omega$, with the sum being locally integrable in $\Omega$. Then the series converges locally uniformly in $\Omega$. Indeed, let us write
$$
\sum_{n=1}^{\infty} \abs{f_n(z)}^2 =: F(z)\in L^{1}_{{\rm loc}}(\Omega).
$$
Let $K\subset \Omega$ be compact and let $\omega$ be an open neighborhood of $K$ such that $K\subset \omega\subset \subset \Omega$. Then by Cauchy's integral formula and the Cauchy-Schwarz inequality we have
$$
\sup_{K} \abs{f_n}^2 \leq {\cal O}_{K,\omega}(1) \norm{f_n}_{L^2(\omega)}^2.
$$
We get therefore the uniform bound
$$
\sum_{n=1}^N  \sup_{K} \abs{f_n}^2 \leq {\cal O}_{K,\omega}(1) \norm{F}_{L^1(\omega)},\quad  N=1,2,\ldots,
$$
implying the locally uniform convergence of (\ref{eqS.7.1}).

\medskip
\noindent
It follows that (\ref{eqS.6}) holds with locally uniform convergence in $x\in \Omega_2$, and in particular the function $\Omega_2 \ni x \mapsto \norm{k_x}_{\Phi_1}^2$ is continuous plurisubharmonic. We may therefore conclude that the series in (\ref{eqS.5}) converges locally uniformly in $\Omega_1 \times \Omega_2$.  Letting
\begeq
\label{eqS.8}
K(x,z) := \overline{k_x(\overline{z})} = \sum_{j=1}^{\infty} Te_j(x) \overline{e_j(\overline{z})},
\endeq
we conclude that $K \in {\rm Hol}(\Omega_2 \times \overline{\Omega_1})$ is such that (\ref{eqS.21}) and (\ref{eqS.22}) hold, and these properties characterize the kernel $K$ uniquely.

\medskip
\noindent
When verifying (\ref{eqS.23}), we let $\widetilde{k}_x\in H_{\Phi_2}(\Omega_2)$ be the reproducing kernel for $H_{\Phi_2}(\Omega_2)$. We may then write,
when $f\in H_{\Phi_1}(\Omega_1)$, $x \in \Omega_2$,
\begeq
\label{eqS.81}
Tf(x) = (Tf, \widetilde{k}_x)_{\Phi_2} = (f, T^* \widetilde{k}_x)_{\Phi_1},
\endeq
and therefore,
\begeq
\label{eqS.9}
k_x = T^* \widetilde{k}_x.
\endeq
Here
$$
T^*: H_{\Phi_2}(\Omega_2) \rightarrow H_{\Phi_1}(\Omega_1)
$$
is the adjoint of $T$. Letting $(f_j)$ be an orthonormal basis for $H_{\Phi_2}(\Omega_2)$ and recalling that
\begeq
\label{eqS.10}
\widetilde{k}_x = \sum_{j=1}^{\infty} \overline{f_j(x)} f_j,
\endeq
we get
\begeq
\label{eqS.11}
k_x(y) = \sum_{j=1}^{\infty} \overline{f_j(x)} T^* f_j(y),
\endeq
Therefore,
$$
K(x,\overline{y}) = \sum_{j=1}^{\infty} f_j(x) \overline{T^* f_j(y)}.
$$
and we see that (\ref{eqS.23}) follows. We also get
\begeq
\label{eqS.12}
\norm{K(\cdot, \overline{y})}_{\Phi_2}^2 = \sum_{j=1}^{\infty} \abs{T^* f_j(y)}^2.
\endeq

\bigskip
\noindent
{\it Remark.} It follows from (\ref{eqS.6}) that $T\in {\cal L}(H_{\Phi_1}(\Omega_1), H_{\Phi_2}(\Omega_2))$ is of Hilbert-Schmidt class precisely when
$$
\int\!\!\! \int_{\Omega_1 \times \Omega_2} \abs{K(x,\overline{y})}^2 e^{-2(\Phi_1(y) + \Phi_2(x))}\, L(dy)\, L(dx) < \infty.
$$

\medskip
\noindent
{\it Remark.} An alternative proof of Theorem \ref{theo_kernel} can be obtained by applying the Schwartz kernel theorem directly to the linear continuous map
$$
\Pi_{\Phi_2} T \Pi_{\Phi_1}: L^2(\Omega_1, e^{-2\Phi_1} L(dy_1)) \rightarrow L^2(\Omega_2, e^{-2\Phi_2} L(dy_2)).
$$
Here
$$
\Pi_{\Phi_j}: L^2(\Omega_j, e^{-2\Phi_j} L(dy_j)) \rightarrow H_{\Phi_j}(\Omega_j)
$$
is the orthogonal projection. Writing the Schwartz kernel of $\Pi_{\Phi_2} T \Pi_{\Phi_1}$ in the form $K(x,\overline{y}) e^{-2\Phi_1(y)}$, we see that $K$ should satisfy $\partial_{\overline{x}} K(x,\overline{y}) = 0$. Now the distribution kernel of the adjoint $\Pi_{\Phi_1} T^* \Pi_{\Phi_2}$ is given by
$\overline{K(y, \overline{x})} e^{-2\Phi_2(y)}$, and it follows that $\partial_{\overline{x}} \left(\overline{K(y, \overline{x})}\right) = 0$. We get
$\partial_x \left(K(y, \overline{x})\right) = 0$, so that $\left(\partial_{\overline{y}} K\right)(y, \overline{x}) = 0 \Longleftrightarrow \partial_{\overline{y}} K(x,y) = 0$. We conclude that $K(x,y)$ is holomorphic in $(x,y)$.

\section{Positivity and Weyl quantization}
\label{appB}
\setcounter{equation}{0}

The purpose of this appendix is to characterize the boundedness of the Weyl quantization of a symbol of the form $\exp (iF(x,\xi))$, where $F$ a
complex quadratic form, in the $H_{\Phi}$-setting. See also~\cite{H95} for a related discussion in the context of $L^2$--boundedness.

\medskip
\noindent
Let $F=F(x,\xi )$ be a complex valued holomorphic quadratic form on ${\bf C}^{2n}$ and let us consider formally the Weyl quantization of $e^{iF(x,\xi )}$, \begin{equation}
\label{rem.0}
Au(x)=\mathrm{Op}^w(e^{iF})u(x)=\frac{1}{(2\pi )^n}\iint e^{i((x-y)\cdot \theta +F((x+y)/2,\theta ))}u(y)dy d\theta .
\end{equation}
The holomorphic quadratic form $(x-y)\cdot \theta +F((x+y)/2,\theta )$ is a non-degenerate phase function in the sense of H\"ormander and generates a canonical relation
\begin{equation}
\label{rem.1}
\kappa :\ (y,\eta )\mapsto (x,\xi ),
\end{equation}
given by
\begin{equation}\label{rem.2}
\begin{split}
x&=\frac{x+y}{2}-\frac{1}{2}F'_\xi (\frac{x+y}{2},\theta ),\\
y&=\frac{x+y}{2}+\frac{1}{2}F'_\xi (\frac{x+y}{2},\theta ),\\
\xi& =\theta +\frac{1}{2}F'_x (\frac{x+y}{2},\theta ),\\
\eta &=\theta -\frac{1}{2}F'_x (\frac{x+y}{2},\theta ).
\end{split}
\end{equation}
The graph is parametrized by $\rho =((x+y)/2,\theta )\in {\bf C}^{2n}$ and (\ref{rem.1}), (\ref{rem.2}) take the form
\begin{equation}
\label{rem.3}
\kappa :\ \rho +\frac{1}{2}H_F(\rho )\mapsto \rho -\frac{1}{2}H_F(\rho ),
\end{equation}
where $H_F(\rho )=(F'_\xi (\rho ),-F'_x(\rho ))$ is the Hamilton field of $F$ at $\rho $.

\medskip
\noindent
We shall now give a criterion for when $\kappa $ in (\ref{rem.3}) is a canonical transformation. Recall that $H_F(\rho) = {\cal F}\rho $, where
$$
{\cal F}= \begin{pmatrix}F''_{\xi x} &F''_{\xi \xi }\\
-F''_{xx} &-F''_{x\xi }\end{pmatrix}
$$
is the fundamental matrix of $F$ (usually appearing as the linearization of a Hamilton vector field, which in our case is already
linear). We have
$$
{\cal F} = JF'',\ J=\begin{pmatrix}0 &1\\ -1 &0\end{pmatrix},\
F'' = \begin{pmatrix}F''_{xx} &F''_{x\xi }\\ F''_{\xi x} &F''_{\xi \xi }\end{pmatrix},
$$
and we notice that $J^2=-1$, $J^\mathrm{t}=-J$. Then (\ref{rem.3}) is the relation
\begin{equation}
\label{rem.3.1}
\left(1 +{\cal F}/2\right)\rho \mapsto \left(1-{\cal F}/2 \right)\rho
\end{equation}
Now ${\cal F}$ is antisymmetric with respect to the bilinear form
$\sigma (\nu  , \mu )= J\nu \cdot \mu $, hence $1-{\cal F}/2$ is
bijective if and only if its transpose $1+{\cal F}/2$ with respect to $\sigma $
is bijective. We conclude that the following three statements are equivalent:
\begin{equation}
\label{rem.3.2}
\begin{split}
\hbox{(i)}\ \ &\kappa \hbox{ is a canonical transformation,}\\
\hbox{(ii)}\  \ &1-{\cal F}/2\hbox{ is bijective},\\
\hbox{(iii)}\ \ &1+{\cal F}/2\hbox{ is bijective}.
\end{split}
\end{equation}
In what follows, we shall assume that (\ref{rem.3.2}) holds.

\bigskip
\noindent
Let $\Phi_0$ be a strictly plurisubharmonic quadratic form on ${\bf C}^n$ and let $\iota_{\Phi_0}:{\bf C}^{2n}\to {\bf C}^{2n}$ be the
corresponding antilinear involution, i.e.\ the unique antilinear map which is equal to the identity on $\Lambda_{\Phi_0}$. We shall now proceed to characterize the positivity of the canonical transformation $\kappa$ in (\ref{rem.3}) relative to $\Lambda_{\Phi_0}$. Let
$$
[\mu ,\nu ] = \frac{1}{2} b(\mu,\nu),
$$
where $b(\mu,\nu)$ has been defined in (\ref{eq2.3.2}). It is a Hermitian form and that $\kappa$ is positive relative to $\Lambda_{\Phi_0}$ precisely when
\begin{equation}\label{rem.4}
[\nu ,\nu ]\ge [\mu ,\mu ], \hbox{ for all }\nu ,\mu  \hbox{ with }\nu = \kappa (\mu ).
\end{equation}
By (\ref{rem.3}) this is equivalent to
$$
[\rho -(1/2)H_F(\rho ), \rho -(1/2)H_F(\rho )]\ge [\rho +(1/2)H_F(\rho
), \rho +(1/2)H_F(\rho )],\ \rho \in {\bf C}^{2n},
$$
or equivalently,
\begin{equation}\label{rem.5}
\Re [H_F(\rho ),\rho ]\le 0,\ \rho \in {\bf C}^{2n}.
\end{equation}

\medskip
\noindent
To simplify the following discussion, we shall make use of the invariance (exact Egorov theorem) under conjugation of $A$ in (\ref{rem.0}) with a unitary metaplectic Fourier integral operator $U:\, L^2({\bf R}^n)\to H_{\Phi_0}({\bf C}^n)$ with the associated canonical transformation $\kappa _U$, mapping
${\bf R}^{2n}$ onto $\Lambda_{\Phi_0}$. The operator $B=U^{-1}AU$ is the Weyl quantization of $e^{iG}$, where $G=F\circ \kappa _U$. Also
$\iota_{\Phi_0} =\kappa _U {\cal C} \kappa _U^{-1}$, where ${\cal C}$ is the involution associated to ${\bf R}^{2n}$, which is just the map of ordinary
complex conjugation. By abuse of notation we write $F$ also for the pull back $F\circ \kappa _U$ and we continue the discussion in the
case when $\Lambda_{\Phi_0}$ has been replaced with ${\bf R}^{2n}$ and $\iota_{\Phi_0}$ with ${\cal C}$, ${\cal C}(\rho )=\overline{\rho} $. In
this setting, (\ref{rem.5}) becomes
$$
\Im \sigma (F'_\xi (\rho ),-F'_x(\rho ); \overline{x},\overline{\xi
})\le 0,\ \forall \rho =(x,\xi )\in {\bf C}^{2n},
$$
i.e.
$$
\Im (F'_x(x,\xi )\cdot \overline{x}+F'_\xi (x,\xi )\cdot \overline{\xi
})\ge 0,\ (x,\xi )\in {\bf C}^{2n},
$$
or even more simply,
$$
\Im (F''_{\rho \rho }\rho \cdot \overline{\rho })\ge 0 .
$$
Writing $\rho =\mu +i\nu $, $\mu ,\nu \in {\bf R}^{2n}$ we see that the
last inequality is equivalent to
$$\Im F''\mu \cdot \mu  +\Im F''\nu \cdot \nu \ge 0,$$
i.e.
$$
\Im F''\ge 0,
$$
i.e. $$\Im F\ge 0\hbox{ on } {\bf R}^{2n}.$$
By the metaplectic invariance it follows that the positivity condition (\ref{rem.4}) is equivalent to
\begin{equation}
\label{rem.6}
\Im F\ge 0 \hbox{ on } \Lambda_{\Phi_0},
\end{equation}
now with the original $F$.

\medskip
\noindent
{\it Remark}. The condition (\ref{rem.6}) is quite natural since we know that for ordinary symbols instead of $e^{iF}$, the natural contour of
integration in (\ref{rem.0}) should be $\theta =(2/i)\partial _x\Phi ((x+y)/2)$, see~\cite{Sj96},~\cite{HiSj15}.

\bigskip
\noindent
We summarize the discussion in this section in the following result.
\begin{prop}
\label{prop_Weyl}
Let $F$ be a holomorphic quadratic form on $\comp^{2n}$ such that the fundamental matrix of $F$ does not have the eigenvalues $\pm 2$. Let $\Phi_0$ be a strictly plurisubharmonic quadratic form on $\comp^n$. The canonical transformation associated to the Fourier integral operator ${\rm Op}^w(e^{iF})$ is positive relative to $\Lambda_{\Phi_0}$ precisely when
\begeq
\label{rem.7}
{\rm Im}\, F|_{\Lambda_{\Phi_0}}\geq 0.
\endeq
In particular, if {\rm (\ref{rem.7})} holds, then the operator
$$
{\rm Op}^w(e^{iF}): H_{\Phi_0}(\comp^n) \rightarrow H_{\Phi_0}(\comp^n)
$$
is bounded.
\end{prop}

\end{appendix}

\end{document}